\documentclass[11pt]{amsart}
\usepackage{amsmath, amssymb, amscd, mathrsfs, url, pinlabel,verbatim}
\usepackage[pagebackref]{hyperref}
\usepackage[margin=1in,marginparwidth=0.75in,centering,letterpaper,dvips]{geometry}
\usepackage{color,dcpic,latexsym,graphicx,epstopdf,comment}
\usepackage[all]{xy}
\usepackage[dvipsnames]{xcolor}
\usepackage{tikz,tikz-cd,pgfplots}
\usepackage{upquote,tabularx,textcomp}
\usepackage[shortlabels]{enumitem}
\usepackage{mathtools}
\usepackage[color=blue!20!white,textsize=tiny]{todonotes}

\title{Pseudo-Anosov flows on hyperbolic L-spaces}

\author[John A. Baldwin]{John A. Baldwin}
\address{Department of Mathematics \\ Boston College}
\email{john.baldwin@bc.edu}

\author[Steven Sivek]{Steven Sivek}
\address{Department of Mathematics\\Imperial College London}
\email{s.sivek@imperial.ac.uk}

\author[Jonathan Zung]{Jonathan Zung}
\address{Department of Mathematics \\ Massachusetts Institute of Technology}
\email{jzung@mit.edu}

\makeatletter
\newtheorem*{rep@theorem}{\rep@title}
\newcommand{\newreptheorem}[2]{%
\newenvironment{rep#1}[1]{%
 \def\rep@title{#2 \ref{##1}}%
 \begin{rep@theorem}}%
 {\end{rep@theorem}}}
\makeatother

\newtheorem {theorem}{Theorem}
\newreptheorem{theorem}{Theorem}
\newtheorem {lemma}[theorem]{Lemma}
\newtheorem {proposition}[theorem]{Proposition}
\newtheorem {corollary}[theorem]{Corollary}

\newtheorem {question}[theorem]{Question}

\numberwithin{equation}{section}
\numberwithin{theorem}{section}

\theoremstyle{definition}

\newtheorem{remark}[theorem]{Remark}
\newtheorem*{remark*}{Remark}

\newlist{pcases}{enumerate}{1}
\setlist[pcases]{
  label=\bf{Case~\arabic*:}\protect\thiscase.~,
  ref=\arabic*,
  align=left,
  labelsep=0pt,
  leftmargin=0pt,
  labelwidth=0pt,
  parsep=0pt
}
\newcommand{\case}[1][]{%
  \if\relax\detokenize{#1}\relax
    \def\thiscase{}%
  \else
    \def\thiscase{~#1}%
  \fi
  \item
}

\newcommand{\Z}{\mathbb{Z}}

\newcommand{\R}{\mathbb{R}}

\newcommand{\Q}{\mathbb{Q}}

\newcommand{\rank}{\operatorname{rank}}

\newcommand{\spectrum}{\mathcal P}
\renewcommand{\phi}{\varphi}

\newcommand{\CH}{\mathit{CH}}

\DeclareFontFamily{U}{mathx}{\hyphenchar\font45}
\DeclareFontShape{U}{mathx}{m}{n}{
      <5> <6> <7> <8> <9> <10>
      <10.95> <12> <14.4> <17.28> <20.74> <24.88>
      mathx10
      }{}
\DeclareSymbolFont{mathx}{U}{mathx}{m}{n}
\DeclareFontSubstitution{U}{mathx}{m}{n}
\DeclareMathAccent{\widecheck}{0}{mathx}{"71}

\newcommand{\hfhat}{\widehat{\mathit{HF}}}

\makeatletter
\DeclareFontFamily{OMX}{MnSymbolE}{}
\DeclareSymbolFont{MnLargeSymbols}{OMX}{MnSymbolE}{m}{n}
\SetSymbolFont{MnLargeSymbols}{bold}{OMX}{MnSymbolE}{b}{n}
\DeclareFontShape{OMX}{MnSymbolE}{m}{n}{
    <-6>  MnSymbolE5
   <6-7>  MnSymbolE6
   <7-8>  MnSymbolE7
   <8-9>  MnSymbolE8
   <9-10> MnSymbolE9
  <10-12> MnSymbolE10
  <12->   MnSymbolE12
}{}
\DeclareFontShape{OMX}{MnSymbolE}{b}{n}{
    <-6>  MnSymbolE-Bold5
   <6-7>  MnSymbolE-Bold6
   <7-8>  MnSymbolE-Bold7
   <8-9>  MnSymbolE-Bold8
   <9-10> MnSymbolE-Bold9
  <10-12> MnSymbolE-Bold10
  <12->   MnSymbolE-Bold12
}{}

\let\llangle\@undefined
\let\rrangle\@undefined
\DeclareMathDelimiter{\llangle}{\mathopen}%
                     {MnLargeSymbols}{'164}{MnLargeSymbols}{'164}
\DeclareMathDelimiter{\rrangle}{\mathclose}%
                     {MnLargeSymbols}{'171}{MnLargeSymbols}{'171}
\makeatother

\newcounter{desccount}

\newcommand{\descref}[1]{\hyperref[#1]{#1}}

\usetikzlibrary{calc,intersections}
\tikzset{every picture/.style=thick}
\tikzset{link/.style = { white, double = black, line width = 1.75pt, double distance = 1.25pt, looseness=1.75 }}
\tikzset{crossing/.style = {draw, circle, dotted, minimum size=0.5cm, inner sep=0, outer sep=0}}
\pgfplotsset{compat=1.12}

\begin{document}

\begin{abstract} We prove that for each  $n\in\mathbb{N}$ there is a hyperbolic L-space with $n$ pseudo-Anosov flows, no two of which are orbit equivalent. These flows have no perfect fits and are thus quasigeodesic. In addition, our flows admit positive Birkhoff sections, which we argue implies that they give rise to $n$ universally tight contact structures whose lifts to any finite cover are non-contactomorphic. This argument involves cylindrical contact homology together with the work of Barthelm\'e, Frankel, and Mann on the reconstruction of pseudo-Anosov flows from their closed orbits. These results answer more general versions of questions posed by Calegari and by Min and Nonino.
\end{abstract}

\maketitle

\section{Introduction}
Two flows on a closed 3-manifold are \emph{orbit equivalent} if there is an orientation-preserving homeomorphism of the manifold sending  orbits of one flow to those of the other. It is a central problem to understand  how the topology  of a 3-manifold is related to the kinds and numbers of flows that it can support. For example, Problem 3.53 of Kirby's list \cite{kirby-list} asks whether there exists for every $n\in \mathbb{N}$  a closed hyperbolic 3-manifold with $n$ orbit inequivalent Anosov flows. 

There has been a flurry of recent work related to this problem. In \cite{bby}, B\'eguin, Bonatti, and Yu found the first examples of closed 3-manifolds with arbitrarily many orbit inequivalent Anosov flows, but their examples were toroidal.  Then in  \cite{bowden-mann}, Bowden and Mann did the same for closed hyperbolic 3-manifolds, solving the Kirby problem above.  Clay and Pinsky later found additional, simpler toroidal examples in \cite{clay-pinsky}, obtained by gluing together two trefoil complements.

Such flows are closely related with taut foliations. Indeed, every closed 3-manifold with an Anosov flow has a taut foliation, as does any 3-manifold supporting a pseudo-Anosov flow without odd-pronged singular orbits.\footnote{In this paper, we do not assume that taut foliations are co-orientable unless that is explicitly stipulated.} Conversely, the existence of a co-orientable taut foliation is conjectured to imply the existence of a pseudo-Anosov flow.  It is therefore natural to ask the analogue of the Kirby question  for closed hyperbolic 3-manifolds with no taut foliations, where \emph{Anosov} is (necessarily) replaced by \emph{pseudo-Anosov}. Our main result is an affirmative answer to this question. 

We will describe this result below, and then explain how it resolves more general versions of a question of Calegari on  quasigeodesic flows and a conjecture of Min and Nonino about universally tight contact structures on hyperbolic L-spaces.

For each $n\in \mathbb{N}$, let $\mathbb{L}_n = L_0  \cup \dots \cup L_{n-1}\subset S^3$ be the oriented chain link shown in Figure~\ref{fig:chainlink}. Given an ordered $n$-tuple $(r_0,\dots,r_{n-1})$ of rational numbers, let $\mathbb{L}_n(r_0,\dots,r_{n-1})$ be the 3-manifold obtained  by performing $r_i$-surgery on $L_i$ for  each $i = 0,\dots n-1$. Our main theorem is:

\begin{theorem}
\label{thm:main}
For each even $n\geq 4$, there are infinitely many  $(r_0,\dots,r_{n-1})\in \big(\frac{1}{4}\Z\big)^n$ such that \[\mathbb{L}_n(r_0,\dots,r_{n-1})\] is a hyperbolic L-space with $n$ distinct pseudo-Anosov flows, no two of which are orbit equivalent. In particular, these manifolds have no taut foliations.
\end{theorem}

\begin{figure}[ht]
\begin{tikzpicture}
\begin{scope}
\clip (-3,-1.2) -- (3,-1.2) -- (3,-0.6) -- (2.75,-0.6) -- (2.75,0.8) -- (-2.75,0.8) -- (-2.75,-0.6) -- (-3,-0.6) -- cycle;
\draw[link,double=red] (0,0) ellipse (0.75 and 0.5) ++(0,-0.5) node[below,red] {$L_0$};
\foreach \i/\lname in {-2/2,-1/1,1/n-1,2/n-2} {
  \draw[link] (1.25*\i,0) ellipse (0.75 and 0.5) ++(0,-0.5) node[below,black] {$L_{\lname}$};
}
\foreach \i in {-2,-1,0,1} {
  \ifthenelse{\i=0}{\def\lcolor{red}}{\def\lcolor{black}}
  \draw[link, double=\lcolor] (1.25*\i,0) ++(0,-0.5) arc (270:360:0.75 and 0.5);
  \draw[link, double=\lcolor] (1.25*\i,0) ++(0,0.5) arc (90:180:0.75 and 0.5);
  \draw[line width=0.8pt, \lcolor, -Stealth] (1.25*\i,0) ++ (90:0.75 and 0.5) -- ++(-0.1,0);
}
\draw[line width=0.8pt, black, -Stealth] (1.25*2,0) ++ (90:0.75 and 0.5) -- ++(-0.1,0);
\end{scope}
\foreach \i in {-1,1} { \node at (\i*3.125,0) {$\cdots$}; }
\begin{scope}
\clip (3.5,0.8) -- (3.5,-0.8) -- (5,-0.8) -- (5,2.5) -- (-5,2.5) -- (-5,-0.8)  -- (-3.5,-0.8) -- (-3.5,0.8) -- cycle;
\foreach \i in {-1,1} {
  \draw[link] (3.75*\i, 0) ellipse (0.75 and 0.5);
  \draw[line width=0.8pt, black, -Stealth] (3.75*\i,0) ++ (90:0.75 and 0.5) -- ++(-0.1,0);
  \draw[link] (5*\i,0) ++(0,0.5) arc (90:90+180*\i:0.75 and 0.5);
  \draw[link] (5*\i,0) ++ (-1.25*\i,0) ++ (-120*\i:0.75 and 0.5) arc (-120*\i:90-90*\i:0.75 and 0.5);
}
\end{scope}
\draw[link] (5,0.5) to[out=0,in=0] ++(0,1) -- ++(-10,0) to[out=180,in=180] ++(0,-1);
\draw[link] (5,-0.5) to[out=0,in=0] ++(0,3) -- ++(-10,0) to[out=180,in=180] ++(0,-3);
\draw[line width=0.8pt, black, -Stealth] (0,1.5) -- ++(0.1,0);
\end{tikzpicture}
\caption{The chain link $\mathbb{L}_n$.}
\label{fig:chainlink}
\end{figure}
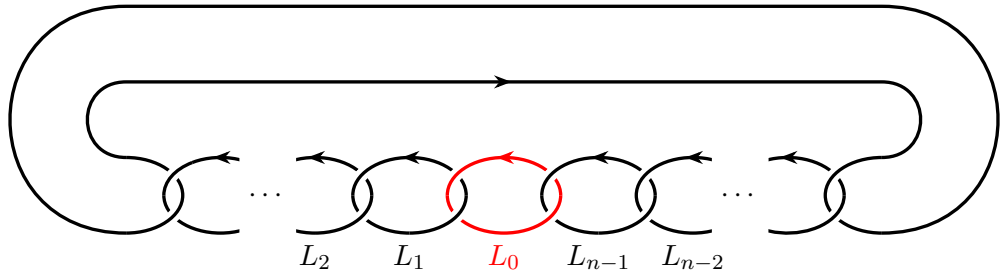

The link $\mathbb{L}_n$ is known to be hyperbolic \cite{neumann-reid} and an L-space link \cite{liu-lspace-links} for $n\geq 3$. The latter means that all sufficiently large integer surgeries on $\mathbb{L}_n$ are L-spaces; we show in \S\ref{sec:rational-surgeries} that large rational surgeries on L-space links are also L-spaces. In particular, infinitely many quarter-integer surgeries on $\mathbb{L}_n$ are hyperbolic L-spaces. Our pseudo-Anosov flows arise via Fried surgery on a fixed pseudo-Anosov flow in the complement of $\mathbb{L}_n$, corresponding to a certain fibered face of its Thurston unit norm ball. This flow has the property that its degeneracy slopes on the boundary tori are not preserved by the natural rotational symmetry of $\mathbb{L}_n$. We use this to find quarter-integer slopes for which the induced  flows on the Dehn surgery and its homeomorphic  rotations are distinguished by the number of prongs  at various closed orbits, and are thus orbit inequivalent. Behind the scenes is the  result of Shannon \cite{shannon} and its generalization as explained by Agol--Tsang \cite{agol-tsang}, which says that a transitive topological pseudo-Anosov flow has an orbit equivalent smooth model.

Some remarks are warranted about the claim that these quarter-integer L-space surgeries on $\mathbb{L}_n$ have no taut foliations. It is well-known that L-spaces do not admit \emph{co-orientable} taut foliations, but it is not as well-known that they can have taut foliations which are not co-orientable. Indeed, $56/3$-surgery on $P(-2,3,7)$ is a hyperbolic L-space with an Anosov flow, and hence a taut foliation. 

On the other hand, an L-space with odd-order first homology does not admit any taut foliation, co-orientable or otherwise: if it had a taut foliation that was not co-orientable, then some connected double cover would have a co-orientable taut foliation, but odd-order first homology implies that every double cover is disconnected. The rest of Theorem~\ref{thm:main} follows after noting that quarter-integer surgery on a link always has odd-order first homology.

This discussion inspires the question of whether there are closed hyperbolic 3-manifolds without \emph{co-orientable} taut foliations with arbitrarily many inequivalent Anosov flows, or similarly:

\begin{question}
Does there exist, for each $n\in \mathbb{N}$, a hyperbolic L-space with $n$ orbit inequivalent Anosov flows? Any such L-space must have even-order first homology by the discussion above.
\end{question}

In fact, the 3-manifolds with many inequivalent Anosov flows found in \cite{bby,bowden-mann,clay-pinsky} all have $b_1>0$, leading to the even simpler open question:

\begin{question}
Does there exist, for each $n\in \mathbb{N}$, a hyperbolic rational homology 3-sphere with $n$ orbit inequivalent Anosov flows? 
\end{question}

\subsection{Quasigeodesic flows and universally tight contact structures} Motivated by the close relationships between taut foliations and quasigeodesic flows, Calegari proposed the following question for the new K3 problem list: is there a closed hyperbolic 3-manifold with a quasigeodesic flow but no taut foliation?  In a different direction, Min and Nonino conjectured in \cite[Conjecture 1.1]{min-nonino} that no hyperbolic L-space admits a universally tight contact structure. 
The following theorem answers Calegari's question affirmatively and refutes Min--Nonino's conjecture:

\begin{theorem} \label{thm:knot}
Suppose that $K$ is a hyperbolic L-space knot of genus $g$, and $r=p/q$ is a rational number greater than $4g$. Then  $S^3_{r}(K)$ is a hyperbolic L-space which admits:
\begin{itemize}
\item a universally tight contact structure, and 
\item pseudo-Anosov flow with no perfect fits (hence, a quasigeodesic flow).
\end{itemize}
In particular, when $p$ is odd, $S^3_{r}(K)$ admits a quasigeodesic flow but no taut foliation.
\end{theorem}

This result has been known to the authors (and probably others) for some time, modulo Shannon's result above and Agol--Tsang's adaptation to the pseudo-Anosov case, but does not appear to have been widely appreciated. The existence of universally tight contact structures follows from work of Colin and Honda  \cite{colin-honda}. The result about flows comes from Fried surgery, together with Fenley's result \cite{fenley} that pseudo-Anosov flows with no perfect fits are quasigeodesic. 

The construction behind our Theorem \ref{thm:main} resolves more general variants of Calegari's question and Min and Nonino's conjecture. Indeed, the flows in our main theorem have no perfect fits, which immediately implies the following by Fenley's work:

\begin{theorem}
\label{thm:qg}
For each $n\in \mathbb{N}$, there are infinitely many closed hyperbolic 3-manifolds with no taut foliation, but with $n$ orbit inequivalent quasigeodesic pseudo-Anosov flows.
\end{theorem}

The pseudo-Anosov flows in Theorem \ref{thm:main}  also have negative Birkhoff sections. It follows that the corresponding flows on the orentation-reversed manifolds 
\[-\mathbb{L}_n(r_0,\dots,r_{n-1})\]
have positive Birkhoff sections, and thus give rise to universally tight contact structures by the work of the third author \cite{zung}. We show that these contact structures admit contact 1-forms whose Reeb flows are orbit equivalent to the corresponding pseudo-Anosov flows. Inspired by the proof of \cite[Theorem~1.10]{barthelme-mann}, we then use the invariance of cylindrical contact homology, together with the orbit inequivalence of our pseudo-Anosov flows, to argue that these contact structures are non-contactomorphic, and remain so in every finite cover, leading to:

\begin{theorem}
\label{thm:UT}
For each $n\in\mathbb{N}$, there are infinitely many hyperbolic L-spaces with $n$ universally tight contact structures whose lifts to any finite cover are pairwise non-contactomorphic.
\end{theorem}

\begin{remark}
In \cite{bowden-mann}, Bowden and Mann construct hyperbolic 3-manifolds with arbitrarily many pairwise distinct Reeb Anosov flows. These flows are distinguished by the set $\mathcal P(\phi)$ of free homotopy classes of the closed orbits of a flow $\phi$.  This invariant is well-behaved under finite covers because if $G \subset \pi_1(Y)$ is a finite-index subgroup then $P(\phi) \cap G$ determines $P(\phi)$, as a consequence of \cite[Lemma~4.2]{zung}.  Barthelm\'e and Mann \cite{barthelme-mann} showed that $\mathcal P(\phi)$ determines $\phi$ up to \emph{isotopy equivalence}  (i.e., orbit equivalence via a homeomorphism isotopic to the identity), so it follows that the Anosov flows from \cite{bowden-mann} remain isotopically inequivalent in finite covers; then they and Bowden use cylindrical contact homology to show that the underlying contact structures are non-isotopic, and by the same argument they remain so in finite covers.  On the other hand, since $G$ might have more symmetries than $\pi_1(Y)$, it is not clear whether the flows remain orbit inequivalent in finite covers.
\end{remark}

Colin, Giroux, and Honda \cite{colin-giroux-honda} proved that a closed, atoroidal 3-manifold admits only finitely many mutually non-isotopic tight contact structures.  We remark that in combination with the work needed to prove Theorem~\ref{thm:UT}, this result leads to an unconditional proof of \cite[Theorem~B]{zung} (and subsequently also \cite[Theorem~C]{zung}) for atoroidal rational homology 3-spheres, without relying on the foundations of symplectic field theory:

\begin{theorem} \label{thm:finite-birkhoff-section}
Let $M$ be a closed, oriented, atoroidal rational homology 3-sphere.  Then there are only finitely many pseudo-Anosov flows on $M$ that admit positive Birkhoff sections, up to orbit equivalence via a homeomorphism isotopic to the identity.
\end{theorem}

Lens spaces admit both positive and negative universally tight contact structures, but it would be interesting, especially in light of Massoni's recent work \cite{massoni}, to answer the following modification of Min and Nonino's question:

\begin{question}
Is there a hyperbolic L-space with both positive and negative universally tight contact structures? 
\end{question}

One could further ask whether there are hyperbolic L-spaces with positive and negative universally tight contact structures in the same homotopy class, or with arbitrarily many universally tight contact structures of each sign.

\subsection{Organization}
In \S\ref{sec:prelim}, we recall some basic notions and results involving pseudo-Anosov flows. We also extend the third author's previous work to prove that orbit inequivalent pseudo-Anosov flows with positive Birkhoff sections on a hyperbolic rational homology 3-sphere give rise to non-contactomorphic universally tight contact structures (see Propositions \ref{prop:contactform} and \ref{prop:rsft}), and establish Theorem~\ref{thm:finite-birkhoff-section} along the way. In \S\ref{sec:proofs}, we prove the other results stated in this introduction.

\section{Preliminaries} \label{sec:prelim}
We first review some basic notions and results about pseudo-Anosov flows. We then discuss and prove some results relating pseudo-Anosov flows and contact structures. Finally, we narrow our focus to pseudo-Anosov flows on Dehn surgeries on fibered hyperbolic links.
\subsection{Basic notions}
Two flows on a 3-manifold $M$ are \emph{orbit equivalent} if there is a homeomorphism of $M$ sending orbits of one flow to orbits of the other, preserving the orientations of the flowlines but not necessarily the parametrizations of the flows.

An \emph{Anosov flow} on a 3-manifold $M$ is a smooth flow $\phi^t$ which preserves a continuous splitting of the tangent bundle $TM = E^s\oplus E^u \oplus X$ such that $d\phi^t/dt$ spans $X$, and the time-$t$ flow uniformly contracts $E^s$ and uniformly expands $E^u$. That is, there exist constants $C,k>0$ such that
\[ |D\phi^t (v)| \leq C e^{-kt} |v| \]
for any $v\in E^s$ and
\[ |D\phi^t (v)| \geq C e^{kt} |v| \]
for any $v\in E^u$, for some Riemannian metric on $M$. 

A \emph{pseudo-Anosov flow} is a flow which is Anosov except at finitely many singular orbits, where it is locally modelled on the $\frac k 2$-fold branched cover over an orbit of an Anosov flow for some $k\geq 3$. We say that the flow has a \emph{$k$-prong singularity} near such an orbit.

One may also define the weaker notion of a \emph{topological Anosov flow}. We say that a flow $\phi^t$ on a 3-manifold $M$ is a topological Anosov flow if:
\begin{enumerate}
	\item the flow lines $t \mapsto \phi^t(x)$ are $C^1$ and non-constant,
	\item there exist transverse (possibly non-orientable) codimension 1 foliations $\mathscr F^s$ and $\mathscr F^u$ of $M$ preserved by the flow of $\phi$,
	\item and flowlines in leaves of $\mathscr F^s$ converge in forward time and flowlines in leaves of $\mathscr F^u$ converge in backward time.
\end{enumerate}
A \emph{topological pseudo-Anosov flow} is a flow satisfying the same properties as a topological Anosov flow, except that $\mathscr F^s$ and $\mathscr F^u$ are permitted to have singularities at finitely many orbits $\{\gamma_i\}$. We require that $\mathscr F^s$ and $\mathscr F^u$ each have $k_i \geq 3$ half-leaves meeting along each singular orbit $\gamma_i$. We call $k_i$ the number of \emph{prongs of $\phi$ at $\gamma_i$}.

A flow is said to be \emph{transitive} if it has a dense orbit. Every pseudo-Anosov flow is a topological pseudo-Anosov flow---the plane fields $E^s\oplus X$ and $E^u\oplus X$ are integrable (away from singular orbits) and integrate to the foliations $\mathscr F^s$ and $\mathscr F^u$. Conversely, every transitive topological pseudo-Anosov flow is orbit equivalent to a pseudo-Anosov flow; see the recent work of Shannon for the Anosov case \cite{shannon}, and the discussion in \cite[\S5.5]{agol-tsang} for the adaptations necessary in the pseudo-Anosov case. This fact will be useful since Fried surgery, described below, a priori produces only a topological pseudo-Anosov flow. On the other hand, topological pseudo-Anosov flows lack the smoothness and structural stability properties which are needed for the results of \cite{zung}, which we will apply in order to prove Propositions \ref{prop:contactform} and \ref{prop:rsft} relating pseudo-Anosov flows and contact structures.

Suppose $\phi$ is a topological pseudo-Anosov flow on $M$. A closed disk $R$ immersed in $Y$ is called a \emph{rectangle} if it is transverse to $\phi$ and the induced foliations $\mathscr F^s|_R$ and $\mathscr F^u|_R$ form a product foliation. A rectangle may intersect singular orbits on its boundary, but not in its interior. An \emph{admissible homotopy} of a rectangle is a homotopy which moves points along flowlines of $\phi$. We say that $R_1$ is \emph{contained} in $R_2$ if there is an admissible homotopy taking $R_1$ into a subrectangle of $R_2$. Finally, we say that $\phi$ has \emph{no perfect fits} if every ascending chain of rectangles $R_1 \subset R_2 \subset \dots$ has an upper bound---i.e., a rectangle $R$ containing every element of the chain. The following result of Fenley \cite[Theorem~F]{fenley} indicates the importance of this notion:

\begin{theorem}[{\cite{fenley}}] \label{thm:fenley}
A topological pseudo-Anosov flow with no perfect fits is quasigeodesic.
\end{theorem}

A \emph{Birkhoff section} for a flow $\phi$ on $M$ is a compact, oriented surface in $M$ whose interior is embedded and transverse to $\phi$, and whose boundary consists of closed orbits of $\varphi$ (possibly winding around orbits multiple times, and with forwards or backwards orientation), and which intersects every orbit of $\phi$ in forwards and backwards time. A Birkhoff section is \emph{positive} (resp.~\emph{negative}) if all of its boundary components are oriented with (resp.~against) the flow. 

\begin{remark}
\label{rmk:posneg}
If a flow $\varphi$ on $M$ has a negative Birkhoff section $F$, then $-F$ is a positive Birkhoff section for the corresponding flow on $-M$.
\end{remark}

\subsection{Flows and contact structures}
\label{ssec:tight}

A positive Birkhoff section for a pseudo-Anosov flow $\phi$ on $M$ gives rise to a rational open book decomposition \cite{bevhm}, and hence to a uniquely defined isotopy class of contact structures on $M$ \cite[Theorem~1.7]{bevhm}. The next proposition says that as long as a homological obstruction vanishes, we can choose a contact form for this contact structure so that the Reeb flow has essentially the same dynamics as $\phi$.

\begin{proposition}\label{prop:contactform}
Suppose that $\phi$ is a pseudo-Anosov flow on a rational homology 3-sphere $M$ with a positive Birkhoff section, and let $\xi$ be the contact structure associated with the Birkhoff section. Then $\xi$ has a contact form $\alpha$ whose Reeb flow $R_\alpha$ satisfies the following properties.
\begin{enumerate}
	  \item Every orbit of $R_\alpha$ representing a primitive homotopy class is nondegenerate, and $R_\alpha$ has no contractible orbits. \label{i:reeb-nondegenerate}
	  \item For each primitive homotopy class $[\gamma] \subset \pi_1(M)$ not represented by a closed orbit of $\phi$, $R_\alpha$ has no closed orbits in $[\gamma]$. \label{i:primitive-orbit-classes}
	  \item For each primitive homotopy class $[\gamma] \subset \pi_1(M)$ represented by a closed orbit of $\phi$, the signed count of representatives of $[\gamma]$ among closed orbits of $R_\alpha$ is nonzero. Here, the sign of a closed orbit of $R_\alpha$ is its Lefschetz index. \label{i:gamma-signed-count}
	  \item $R_\alpha$ is compatible with the given Birkhoff section, i.e., it is transverse to the pages of the Birkhoff section and tangent to its boundary. \label{i:birkhoff-compatible}
\end{enumerate}
If $\pi: \tilde{M} \to M$ is any finite cover, then the same holds for the lifted pseudo-Anosov flow $\tilde\phi = \pi^*\phi$ and the associated contact structure $\tilde\xi$.
\end{proposition}

\begin{proof}
	In \cite[\S3.3]{zung}, Zung constructs a stable Hamiltonian structure $(\omega_0,\lambda_0)$ whose Reeb flow $R_{\omega_0, \lambda_0}$ satisfies \eqref{i:reeb-nondegenerate}, \eqref{i:primitive-orbit-classes}, and \eqref{i:birkhoff-compatible}. (This construction assumes that $\phi$ is transitive, but only in order to conclude that $\varphi$ has a Birkhoff section, which is true here by hypothesis.)  In \cite[Lemma 4.4]{zung} a small exact homotopy is performed, resulting in a stable Hamiltonian structure $(\omega,\lambda)$ whose Reeb flow achieves \eqref{i:gamma-signed-count} while maintaining \eqref{i:reeb-nondegenerate}, \eqref{i:primitive-orbit-classes}, and \eqref{i:birkhoff-compatible}. Since our Birkhoff section may have singular orbits on the boundary, contra the assumption in \cite[\S3.3]{zung}, the conclusion of \cite[Lemma 4.4]{zung} does not hold verbatim.  Instead, the argument shows that for any primitive free homotopy class $[\gamma] \subset \pi_1(M)$ represented by a closed orbit of $\phi$, one of the following possibilities holds:
	\begin{enumerate}[label=(\alph*)]
		\item $R_{\omega,\lambda}$ has a number of hyperbolic orbits representing $[\gamma]$, all of the same Lefschetz index.
		\item $R_{\omega,\lambda}$ has one elliptic orbit representing $[\gamma]$, and no hyperbolic orbits.
		\item $[\gamma]$ is represented by a $k$-pronged orbit such that the first return map does not rotate the prongs. In this free homotopy class, $R_{\omega,\lambda}$ has one elliptic orbit and at least $k\geq 2$ hyperbolic orbits, all of the same Lefschetz index. In particular, the sum of their Lefschetz indices is nonzero. \label{i:multipleorbits}
	\end{enumerate}
	But in any of these cases, $R_{\omega,\lambda}$ satisfies \eqref{i:gamma-signed-count} as desired.
\begin{remark}
In case \ref{i:multipleorbits}, \cite{zung} further cancels the single elliptic orbit with one of the hyperbolic orbits using a ``nonrotating blowup''. However, this requires the $k$-pronged orbit to lie on the interior of the Birkhoff section, and we cannot guarantee this while maintaining positivity of the Birkhoff section.  Since we want the Birkhoff section to remain positive, we leave these orbits as they are.
\end{remark}

We now show that one can perturb $\lambda$ to a contact form $\alpha$ without changing the Reeb flow. Since we started with a positive Birkhoff section, $\lambda$ is almost a contact form in the sense that $\lambda \wedge d\lambda \geq 0$.  Since $\ker(\omega) \subset \ker(d\lambda)$ we can write $d\lambda = f\omega$ for some $f: M \to \R$, and then the nonnegativity of $\lambda \wedge d\lambda = f(\lambda \wedge \omega)$ implies that $f \geq 0$. Since $H^1(M) = 0$ we can fix a primitive $\eta$ for $\omega$, and we consider the 1-form $\alpha = \lambda + \varepsilon \eta$ where $\varepsilon > 0$ is small. Note that
\begin{align*}
\alpha \wedge d\alpha &= (\lambda+\varepsilon\eta) \wedge (d\lambda + \varepsilon \omega) \\
&= (\lambda+\varepsilon\eta) \wedge (f+\varepsilon)\omega \\
&= (f+\varepsilon)\big( \lambda\wedge\omega + \varepsilon(\eta\wedge\omega) \big),
\end{align*}
which is a volume form for small enough $\varepsilon > 0$ because $f+\varepsilon > 0$ and $\lambda\wedge\omega > 0$.  We also have
\[ R_{\omega,\lambda} \in \ker(d\alpha) = \ker(d\lambda + \varepsilon\omega), \]
and if $\varepsilon > 0$ is again small then
  \begin{equation*}
    \alpha(R_{\omega, \lambda}) = \lambda(R_{\omega, \lambda}) + \varepsilon\cdot\eta(R_{\omega, \lambda})>0.
  \end{equation*}
  Together, these facts imply that $\alpha$ is a contact form with Reeb flow parallel to $R_{\omega,\lambda}$. Since $R_\alpha$ is compatible with the same Birkhoff section that defines $\xi$, the form $\alpha$ moreover defines the same contact structure as $\xi$ up to isotopy, completing the proof in this case.

Now if $\tilde\phi$ is the pullback of $\phi$ to some finite cover $\tilde{M} \to M$, then we can lift the positive Birkhoff section on $M$ to one on $\tilde{M}$ and then attempt to carry out the above construction verbatim for $\tilde\phi$.  The only place where we used the hypothesis $b_1(M)=0$ was in constructing a primitive $\eta$ for $\omega$, so it will suffice to show here that $\tilde\omega$ is exact.  The respective constructions on $\tilde{M}$ and $M$ start with stable Hamiltonian structures $(\tilde\omega_0,\tilde\lambda_0)$ and $(\omega_0,\lambda_0)$ constructed using the respective Birkhoff sections; the closed 2-form $\omega_0$ is exact since $b_1(M)=0$, so its pullback $\tilde\omega_0$ is exact as well.  The desired $(\tilde\omega, \tilde\lambda)$ is then obtained from $(\tilde\omega_0,\tilde\lambda_0)$ by an exact homotopy, which means that $[\tilde\omega] = [\tilde\omega_0] = 0$ and hence that $\tilde \omega$ is exact, and we can repeat the rest of the argument for $\tilde\phi$ as claimed.
\end{proof}

The Reeb flow we constructed in Proposition \ref{prop:contactform} has no contractible orbits---in other words, the contact structure $\xi$ is \emph{hypertight}.  Hypertight contact structures are tight \cite[Theorem~1]{hofer}, and their finite covers are again hypertight, so we obtain:
\begin{corollary} \label{cor:universallytight}
	The contact structure associated with a positive Birkhoff section of a pseudo-Anosov flow as in Proposition~\ref{prop:contactform} is universally tight.
\end{corollary}

The proof of the next proposition follows the pattern of \cite[Theorem~1.10]{barthelme-mann}, which deals with the case of skew-Anosov flows.
\begin{proposition}\label{prop:rsft}
 Suppose $\phi_1$ and $\phi_2$ are pseudo-Anosov flows with positive Birkhoff sections on an atoroidal rational homology 3-sphere $M$. Let $\xi_1$ and $\xi_2$ be the associated contact structures.
  \begin{enumerate}
    \item If $\xi_1$ and $\xi_2$ are isotopic, then $\phi_1$ and $\phi_2$ are orbit equivalent via a homeomorphism isotopic to the identity.\label{i:claim2}
    \item More generally, let $\tilde\xi_i$ and $\tilde\phi_i$ be lifts of $\xi_i$ and $\phi_i$ to a finite cover of $M$, for $i=1,2$. If $\tilde\xi_1$ and $\tilde\xi_2$ are isotopic, then $\tilde\phi_1$ and $\tilde\phi_2$ are orbit equivalent via a homeomorphism isotopic to the identity. \label{i:claim3}
  \end{enumerate}
\end{proposition}

\begin{proof}
  Given a flow $\phi$, let $\spectrum(\phi)$ be the set of primitive elements of $\pi_1(M)$ represented by closed orbits of $\phi$. Barthelm\'e, Frankel, and Mann \cite{bfm} prove that $\spectrum(\phi)$ determines a pseudo-Anosov flow up to orbit equivalence isotopic to the identity.  (This claim uses our hypothesis that $M$ is atoroidal, via \cite[Proposition~1.2]{bfm}.  We note that \cite{bfm} generally requires $\phi$ to be transitive, but by \cite[Proposition~2.7]{mosher} this is automatic since $M$ is closed and atoroidal. Finally, note that our definition of $\spectrum(\phi)$ differs slightly from that of \cite{bfm} because it includes only primitive elements. However, the two sets contain the same information because $g^k$ represents a closed orbit if and only if either $g$ or $g^{-1}$ represents an orbit, by \cite[Lemma~4.2]{zung}.)

  Suppose $\xi_1$ and $\xi_2$ are isotopic. Choose contact forms $\alpha_1$ and $\alpha_2$ for $\xi_1$ and $\xi_2$ as in Proposition~\ref{prop:contactform}. These forms are hypertight, so cylindrical contact homology is well-defined as an invariant of each $\xi_i$ up to isotopy   \cite{bao-honda,hutchings-nelson,hutchings-nelson-hypertight}.  In any primitive free homotopy class $[\gamma]$, the summand $\CH(\alpha_i, [\gamma])$ is graded mod 2 by the Conley--Zehnder index, which is even for positive hyperbolic orbits and odd for all others; meanwhile the Lefschetz index of an orbit is $-1$ for a positive hyperbolic orbit and $+1$ otherwise, so if $[\gamma]$ is represented by a closed orbit of $\varphi_i$ then part \eqref{i:gamma-signed-count} of Proposition~\ref{prop:contactform} says that $\chi\big(\CH(\alpha_i,[\gamma])\big)$ is nonzero.  Therefore, $\CH(\alpha_i,[\gamma])$ is nonzero if and only if $[\gamma]$ is represented by a closed orbit of $\phi_i$. In other words, $\CH(\alpha_i)$ determines $\spectrum(\phi_i)$. Since $\CH(\alpha_1)\cong \CH(\alpha_2)$, we conclude from Barthelm\'e--Frankel--Mann's result that $\phi_1$ is orbit equivalent to $\phi_2$ via a homeomorphism isotopic to the identity. This proves \eqref{i:claim2}.

For \eqref{i:claim3}, we observe that Proposition~\ref{prop:contactform} still applies to $\tilde\xi_i$ and $\tilde\phi_i$, because these contact structures and flows are lifted from contact structures and flows on the rational homology sphere $M$.  Thus \eqref{i:claim3} follows by exactly the same argument as \eqref{i:claim2}.
\end{proof}

\begin{proof}[Proof of Theorem~\ref{thm:finite-birkhoff-section}]
Suppose that $M$ is atoroidal and a rational homology sphere, and that $M$ has infinitely many pairwise distinct pseudo-Anosov flows $\phi_i$, each with a positive Birkhoff section.  Then Proposition~\ref{prop:contactform} associates to each $\phi_i$ a contact structure $\xi_i$, which is tight by Corollary~\ref{cor:universallytight}.  Proposition~\ref{prop:rsft} says that no two of the $\xi_i$ are isotopic, so $M$ has infinitely many pairwise distinct tight contact structures, and since $M$ is atoroidal this contradicts a theorem of Colin, Giroux, and Honda \cite[Theorem~2]{colin-giroux-honda}.
\end{proof}

\subsection{Flows and Dehn surgery on hyperbolic fibered links}
\label{ssec:flowsdehn}

We will be specifically interested in pseudo-Anosov flows on Dehn fillings of hyperbolic fibered links. Suppose $L\subset Y$ is an oriented hyperbolic fibered link with components $L_1,\dots,L_n$. Let \[\pi:Y\setminus L\to S^1\] be a fibration of its complement such that the oriented boundary of the closure of each fiber equals $L$. Associated with this fibration is a canonical \emph{suspension} pseudo-Anosov flow $\varphi$. One can extend this flow to Dehn surgeries via an operation called \emph{Fried surgery}, as sketched below.

Let $\nu(L_1),\dots, \nu(L_n)$ be disjoint closed tubular neighborhoods of the link components, and let
\[ \nu(L) = \nu(L_1)\cup \dots \cup \nu(L_n). \]
In Fried surgery, one starts with a \emph{blown up} flow $\bar\varphi$ on $M=Y\setminus \mathring{\nu}(L)$, together with a diffeomorphism
\[ f:M\setminus \partial M \to Y\setminus L \]
identifying $\bar\varphi|_{M\setminus \partial M}$ with $\varphi$. The restriction of $\bar\varphi$ to each torus boundary component $\partial{\nu(L_i)}\subset\partial M$ is a flow with finitely many closed orbits, all of the same slope. These orbits form an oriented multicurve $d_i$ on $\partial{\nu(L_i)}$ whose isotopy class is known as the \emph{degeneracy slope of $\varphi$ at $L_i$}.

Now let $r = (r_1,\dots,r_n)$, where $r_i$ is a slope on $\partial{\nu(L_i)}\subset \partial M$ such that the distance
\[ \Delta(d_i,r_i) := |d_i\cdot r_i| \]
between the degeneracy slope $d_i$ and $r_i$ is at least 2 for each $i =1, \dots, n$. Then $r$-surgery on $L$ can be performed by collapsing the boundary components of $M$ along the slopes $r_1,\dots,r_n$. Moreover, $\bar\varphi$ collapses to a topological pseudo-Anosov flow $\varphi(r)$ on the surgered manifold, which agrees with $\varphi$ (via $f$) away from the cores of the surgery. The core $\gamma_i$ obtained by collapsing $\partial{\nu(L_i)}$ is an orbit of $\varphi(r)$ with $\Delta(d_i,r_i)$ prongs; in particular, it is singular if and only if $\Delta(d_i,r_i)\geq 3$. 

\begin{proposition}\label{prop:surgery-no-perfect-fits}
Suppose $\varphi$ is the suspension pseudo-Anosov flow associated with a fibration of  the complement of a  hyperbolic fibered link $L = L_1\cup \dots \cup L_n$, with degeneracy slopes $d_1,\dots,d_n$ as above. Let $r = (r_1,\dots,r_n)$ be a tuple of boundary slopes such that
\[\Delta(d_i,r_i) \geq 3\]
for each $i$. Then the flow $\varphi(r)$ obtained via Fried surgery has no perfect fits.
\end{proposition}

\begin{proof}
Each core $\gamma_i$ is a singular orbit of the flow $\varphi(r)$ since $\Delta(d_i,r_i) \geq 3$ for all $i$.  Therefore, each rectangle for the flow $\varphi(r)$  avoids every $\gamma_i$ in its interior, and thus gives rise to a rectangle for the flow $\varphi$ (which agrees with $\varphi(r)$ away from the $\gamma_i$). In particular, any  ascending chain of rectangles for $\varphi(r)$ gives rise to a chain of rectangles for $\phi$. This chain in $\phi$ has an upper bound since $\varphi$ is a suspension pseudo-Anosov flow and therefore has no perfect fits by  \cite[Theorem~G]{fenley}, and this upper bound  can be chosen to avoid each $\gamma_i$ in its interior, so the original  chain in $\varphi(r)$ also has an upper bound. Hence, $\varphi(r)$ has no perfect fits.
\end{proof}

We will often reason about the flows $\varphi(r)$ in terms of the fibration $\pi:Y\setminus L\to S^1$. Let $F$ be the closure of a fiber of this fibration, and let $h:F\to F$ be the monodromy, which fixes $\partial F $ pointwise. Then we can identify the mapping torus of $h$, \[\frac{F\times[0,1]}{(x,1)\sim(h(x),0)},\] with $M = Y\setminus \mathring{\nu}(L)$. Let $B_1,\dots,B_n$ be the oriented boundary components of $F$, with $B_i\subset \partial{\nu(L_i)}$. Then $\partial{\nu(L_i)}$ has a meridian-longitude coordinate system, which we denote by $(\mu_i',\lambda_i'),$ where $\lambda_i' = B_i$ and $\mu_i'$ is the oriented meridian of $L_i$ given by \[\mu_i' = p_i\times[0,1]/\sim,\] where $p_i$ is a point in $B_i$. This meridian is oriented so that \[\lambda_i'\cdot\mu_i'=1\] on the oriented boundary component $\partial(M\setminus\mathring{\nu}(L_i)) = -\partial \nu(L_i)$. 

Since $L$ is hyperbolic, the monodromy $h$ is freely isotopic to a pseudo-Anosov homeomorphism $h_0$ of $F$, and the flow $\bar\varphi$ on $M$ may be identified with the suspension flow of $h_0$. The \emph{fractional Dehn twist coefficient of $h$ at $B_i$}, introduced by Honda, Kazez, and Mati{\'c} \cite{hkm-veering2}, is a quantity \[c_{B_i}(h)\in \Q\] which measures the twisting near $B_i$ in the free isotopy from $h$ to $h_0$. It can also be viewed  as a reinterpretation of the degeneracy slope. Indeed, let $q_i$ denote the number of prongs at $B_i$ of the stable foliation of $F$ fixed by $h_0$. Then $c_{B_i}(h) = k_i/q_i$ for some integer $k_i$, and the degeneracy slope $d_i$ is given simply by \begin{equation}\label{eqn:degfdtc}d_i = q_i\mu_i'+k_i\lambda_i',\end{equation} which represents the oriented isotopy class of $\gcd(q_i,k_i)$ copies of some primitive curve on $\partial \nu(L_i)$. We will use this formula to compute degeneracy slopes in \S\ref{sec:proofs}.

Furthermore, the flow $\varphi(r)$ resulting from Fried surgery has a Birkhoff section given by the image of the surface $F$ after collapsing the boundary tori of $M$ along the slopes $r_i$. This Birkhoff section is positive if the oriented curves $d_i$ and $\lambda'_i$ intersect $r_i$ with the same sign for all $i$, and it's negative if $d_i$ and $\lambda'_i$ intersect $r_i$ with the opposite sign for all $i$.

\begin{remark}
\label{rmk:smoothtopological}
The results of \cite{shannon,agol-tsang} discussed above, which assert that a transitive topological pseudo-Anosov flow is orbit equivalent to a pseudo-Anosov flow, only require the flow to be transitive in order to assert the existence of a Birkhoff section \cite{fried,brunella}.  We will apply these results to flows like $\varphi(r)$ that are already known (or assumed) to have Birkhoff sections. In particular, we can and will assume that all of the flows we construct in \S\ref{sec:proofs} are genuine pseudo-Anosov flows rather than merely topological pseudo-Anosov flows.
\end{remark}

\section{Proofs of the main results}\label{sec:proofs}

Fix an even integer $n\geq 4$. Let $\mathbb{L}_n$ be the chain link in Figure \ref{fig:chainlink}, and let $\mathbb{L}_n' = -L_0 \cup L_1 \cup \dots \cup L_{n-1}$ be the oriented link obtained from $\mathbb{L}_n$ by reversing the orientation of the component $L_0$. Then $\mathbb{L}_n'$ is the oriented boundary of the surface $F$ given by a plumbing of one negative horizontal Hopf band with $n-2$ vertical positive Hopf bands and 2 vertical negative Hopf bands, as shown in Figure \ref{fig:F}.
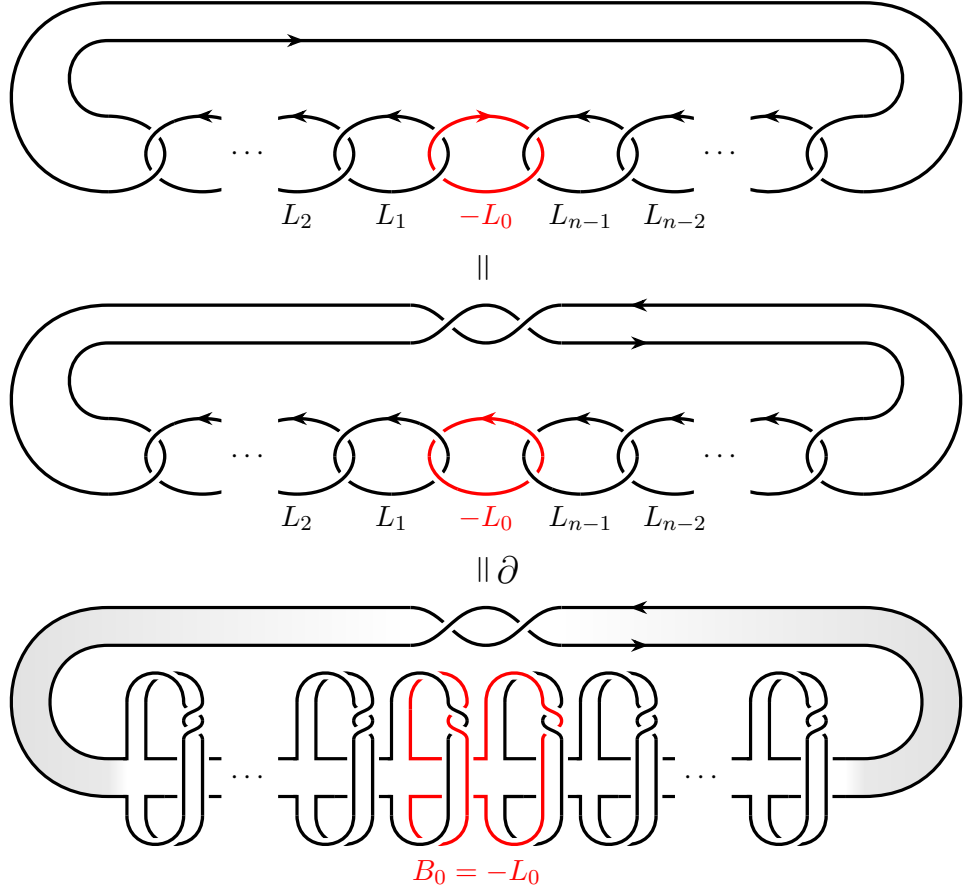
\begin{figure}
\begin{tikzpicture}
\begin{scope}
\begin{scope}
\clip (-3,-1.2) -- (3,-1.2) -- (3,-0.6) -- (2.75,-0.6) -- (2.75,0.8) -- (-2.75,0.8) -- (-2.75,-0.6) -- (-3,-0.6) -- cycle;
\draw[link,double=red] (0,0) ellipse (0.75 and 0.5) ++(0,-0.5) node[below,red] {$-L_0$};
\foreach \i/\lname in {-2/2,-1/1,1/n-1,2/n-2} {
  \draw[link] (1.25*\i,0) ellipse (0.75 and 0.5) ++(0,-0.5) node[below,black] {$L_{\lname}$};
}
\foreach \i in {-2,-1,0,1} {
  \ifthenelse{\i=0}{\def\lcolor{red}}{\def\lcolor{black}}
  \ifthenelse{\i=0}{\def\j{1}}{\def\j{-1}}
  \draw[link, double=\lcolor] (1.25*\i,0) ++(0,-0.5) arc (270:360:0.75 and 0.5);
  \draw[link, double=\lcolor] (1.25*\i,0) ++(0,0.5) arc (90:180:0.75 and 0.5);
  \draw[line width=0.8pt, \lcolor, -Stealth] (1.25*\i,0) ++ (90:0.75 and 0.5) -- ++(0.1*\j,0);
}
\draw[line width=0.8pt, black, -Stealth] (1.25*2,0) ++ (90:0.75 and 0.5) -- ++(-0.1,0);
\end{scope}
\foreach \i in {-1,1} { \node at (\i*3.125,0) {$\cdots$}; }
\begin{scope}
\clip (3.5,0.8) -- (3.5,-0.8) -- (5,-0.8) -- (5,2.5) -- (-5,2.5) -- (-5,-0.8)  -- (-3.5,-0.8) -- (-3.5,0.8) -- cycle;
\foreach \i in {-1,1} {
  \draw[link] (3.75*\i, 0) ellipse (0.75 and 0.5);
  \draw[line width=0.8pt, black, -Stealth] (3.75*\i,0) ++ (90:0.75 and 0.5) -- ++(-0.1,0);
  \draw[link] (5*\i,0) ++(0,0.5) arc (90:90+180*\i:0.75 and 0.5);
  \draw[link] (5*\i,0) ++ (-1.25*\i,0) ++ (-120*\i:0.75 and 0.5) arc (-120*\i:90-90*\i:0.75 and 0.5);
}
\end{scope}
\draw[link] (5,0.5) to[out=0,in=0] ++(0,1) -- ++(-10,0) to[out=180,in=180] ++(0,-1);
\draw[link] (5,-0.5) to[out=0,in=0] ++(0,2.5) -- ++(-10,0) to[out=180,in=180] ++(0,-2.5);
\draw[line width=0.8pt, black, -Stealth] (-2.5,1.5) -- ++(0.1,0);
\end{scope}
\node at (0,-1.5) {\rotatebox{90}{\Large$=$}};
\begin{scope}[yshift=-4cm]
\begin{scope}
\clip (-3,-1.2) -- (3,-1.2) -- (3,-0.6) -- (2.75,-0.6) -- (2.75,0.8) -- (-2.75,0.8) -- (-2.75,-0.6) -- (-3,-0.6) -- cycle;
\draw[link,double=red] (0,0) ellipse (0.75 and 0.5) ++(0,-0.5) node[below,red] {$-L_0$};
\foreach \i/\lname in {-2/2,-1/1,1/n-1,2/n-2} {
  \draw[link] (1.25*\i,0) ellipse (0.75 and 0.5) ++(0,-0.5) node[below,black] {$L_{\lname}$};
}
\foreach \i in {-2,-1,0,1} {
  \ifthenelse{\i=0}{\def\lcolor{red}}{\def\lcolor{black}}
  \ifthenelse{\i=0}{\def\j{1}}{\def\j{-1}}
  \draw[link, double=\lcolor] (1.25*\i,0) ++(0,-0.5) arc (270:360:0.75 and 0.5);
  \draw[link, double=\lcolor] (1.25*\i,0) ++(0,0.5) arc (90:180:0.75 and 0.5);
  \draw[line width=0.8pt, \lcolor, -Stealth] (1.25*\i,0) ++ (90:0.75 and 0.5) -- ++(-0.1,0);
}
\draw[line width=0.8pt, black, -Stealth] (1.25*2,0) ++ (90:0.75 and 0.5) -- ++(-0.1,0);
\end{scope}
\foreach \i in {-1,1} { \node at (\i*3.125,0) {$\cdots$}; }
\begin{scope}
\clip (3.5,0.8) -- (3.5,-0.8) -- (5,-0.8) -- (5,2.5) -- (-5,2.5) -- (-5,-0.8)  -- (-3.5,-0.8) -- (-3.5,0.8) -- cycle;
\foreach \i in {-1,1} {
  \draw[link] (3.75*\i, 0) ellipse (0.75 and 0.5);
  \draw[line width=0.8pt, black, -Stealth] (3.75*\i,0) ++ (90:0.75 and 0.5) -- ++(-0.1,0);
  \draw[link] (5*\i,0) ++(0,0.5) arc (90:90+180*\i:0.75 and 0.5);
  \draw[link] (5*\i,0) ++ (-1.25*\i,0) ++ (-120*\i:0.75 and 0.5) arc (-120*\i:90-90*\i:0.75 and 0.5);
}
\end{scope}
\draw[link,double=red] (0,0) ++ (0.75,0) arc (0:60:0.75 and 0.5);
\draw[link,double=red] (0,0) ++ (-0.75,0) arc (180:240:0.75 and 0.5);
\draw[link] (-1.25,0) ++(0.75,0) arc (0:60:0.75 and 0.5);
\draw[link] (1.25,0) ++(-0.75,0) arc (180:240:0.75 and 0.5);

\draw[link] (5,0.5) to[out=0,in=0] ++(0,1) -- ++(-4,0) ++(-2,0) -- ++(-4,0) to[out=180,in=180] ++(0,-1);
\draw[link] (5,-0.5) to[out=0,in=0] ++(0,2.5) -- ++(-4,0) ++(-2,0) -- ++(-4,0) to[out=180,in=180] ++(0,-2.5);
\draw[link,looseness=1] (-1,2) \foreach \i in {1,2} { to[out=0,in=180] ++(1,-0.5) ++(0,0.5) };
\draw[link,looseness=1] (-1,1.5) \foreach \i in {1,2} { to[out=0,in=180] ++(1,0.5) ++(0,-0.5) };
\draw[line width=0.8pt, black, -Stealth] (2.05,1.5) -- ++(0.1,0);
\draw[line width=0.8pt, black, -Stealth] (2,2) -- ++(-0.1,0);
\end{scope}
\node at (0,-5.5) {\rotatebox{90}{\Large$=$}};
\node[right] at (0,-5.5) {\Large$\partial$};
\begin{scope}[yshift=-8cm]
\shade[looseness=1.75, left color=gray!0, middle color=gray!25, right color=gray!25] (5,0) to[out=0,in=0] ++(0,1.5) -- ++(-4,0) -- ++ (0,0.5) -- ++(4,0) to[out=0,in=0] ++(0,-2.5) -- ++(0,0.5);
\shade[looseness=1.75, left color=gray!25, middle color=gray!25, right color=gray!0] (-5,0) to[out=180,in=180] ++(0,1.5) -- ++(4,0) -- ++ (0,0.5) -- ++(-4,0) to[out=180,in=180] ++(0,-2.5) -- ++(0,0.5);
\begin{scope}
\clip(-5,-0.6) rectangle (5,0.1);
\shade [left color=gray!18, right color=gray!0] (-5.01,0) rectangle (-4.75,-0.5);
\shade [left color=gray!0, right color=gray!18] (5.01,0) rectangle (4.75,-0.5);
\foreach \p in { (-5.002,0), (-5.002,-0.5), (4.75,0), (4.75,-0.5) } { (\draw[line width=1.2pt] \p -- ++(0.252,0); }
\end{scope}
\draw[line width=1.2pt,looseness=1.75] (5,0) to[out=0,in=0] ++(0,1.5) -- ++(-4,0) ++(-2,0) -- ++(-4,0) to[out=180,in=180] ++(0,-1.5);
\draw[line width=1.2pt,looseness=1.75] (5,-0.5) to[out=0,in=0] ++(0,2.5) -- ++(-4,0) ++(-2,0) -- ++(-4,0) to[out=180,in=180] ++(0,-2.5);

\foreach \x in {-4.75,-2.5,-1.25,0,1.25,3.5} {
  \draw[line width=1.2pt] (\x,0) ++(-0.05,0) -- ++(0.05,0) -- ++(0,0.75) coordinate (topbend) ++(0.25,0) -- ++(0,-0.75) -- ++(0.5,0) ++(0.25,0) -- ++(0.25,0) coordinate (rightbend);
  \draw[link] (topbend) ++(0.25,0) to[out=90,in=90] ++(0.75,0) \foreach \i in {1,2} { ++(-0.25,0) to[out=270,in=90,looseness=1] ++(0.25,-0.25) } -- ++(0,-1) to[out=270,in=270] ++(-0.75,0) -- ++(0,0.25) -- ++(0.25,0) coordinate (se-bend);
  \draw[link] (topbend) to[out=90,in=90] ++(0.75,0)  \foreach \i in {1,2} { ++(0.25,0) to[out=270,in=90,looseness=1] ++(-0.25,-0.25) } -- ++(0,-1) to[out=270,in=270] ++(-0.75,0) coordinate (sw-bend);
  \draw[line width=1.2pt] (sw-bend) -- ++(0,0.25) -- ++(-0.05,0);
  \draw[line width=1.2pt] (se-bend) -- ++(0.25,0) ++(0.25,0) -- ++(0.25,0);
  \draw[link] (rightbend) ++ (-0.25,0.25) -- ++(0,-1) ++(-0.25,0) -- ++(0,1);
}
\node at (-3.125,-0.25) {$\cdots$};
\node at (2.875,-0.25) {$\cdots$};
\draw[line width=1.2pt] (-2.5,0) -- ++(-0.25,0) (-2.5,-0.5) -- ++(-0.25,0);
\draw[line width=1.2pt] (3.25,-0) -- ++(0.25,0) (3.25,-0.5) -- ++(0.25,0);
\draw[link,looseness=1] (-1,2) \foreach \i in {1,2} { to[out=0,in=180] ++(1,-0.5) ++(0,0.5) };
\draw[link,looseness=1] (-1,1.5) \foreach \i in {1,2} { to[out=0,in=180] ++(1,0.5) ++(0,-0.5) };
\draw[line width=0.8pt, black, -Stealth] (2.05,1.5) -- ++(0.1,0);
\draw[line width=0.8pt, black, -Stealth] (2,2) -- ++(-0.1,0);

\draw[link,double=red] (-1.25,0) ++ (0.25,0.75) -- ++(0,-0.75) -- ++(0.5,0) ++(0.25,0) -- ++(0.25,0) -- ++(0,0.75) to[out=90,in=90] ++(0.75,0) \foreach \i in {1,-1} { to[out=270,in=90,looseness=1] ++(\i*0.25,-0.25) } -- ++(0,-1) to[out=270,in=270] ++(-0.75,0) -- ++(0,0.25) -- ++(-0.25,0) ++ (-0.25,0) -- ++(-0.5,0) -- ++(0,-0.25) to[out=270,in=270] ++(0.75,0) coordinate (part2);
\draw[link,double=red] (part2) -- ++(0,1) \foreach \i in {-1,1} { to[out=90,in=270,looseness=1] ++(\i*0.25,0.25) } to[out=90,in=90] ++(-0.75,0) -- ++(0,-0.1);
\node[red] at (-0.125,-1.5) {$B_0 = -L_0$};
\draw[link] (-1.25,0) ++(0,-0.75) to[out=270,in=270] ++(0.75,0) -- ++(0,1) ++(0.25,0.25) to[out=90,in=270,looseness=1] ++(-0.25,0.25) to[out=90,in=90] ++(-0.75,0);
\draw[link] (0,0) ++(1,0.25) to[out=90,in=270,looseness=1] ++(-0.25,0.25);

\end{scope}
\end{tikzpicture}
\caption{The chain link $\mathbb{L}'_n$ as the oriented boundary of the surface $F$.}
\label{fig:F}
\end{figure}
In particular, $\mathbb{L}_n'$ is a fibered link with fiber $F$. This fiber is a torus with $n$ disks removed, and the monodromy is the composition \[h = D_{b_0}^{-1}\circ D_{b_1}\circ\dots \circ D_{b_{n-2}} \circ D_{b_{n-1}}^{-1}\circ D_a^{-1} \] of Dehn twists around the curves shown in Figure~\ref{fig:abstract}. 
\begin{figure}
\begin{tikzpicture}
\begin{scope}
\fill[looseness=1.75, gray!15] (5,0) to[out=0,in=0] ++(0,1.5) -- ++(-5,0) -- ++ (0,0.5) -- ++(5,0) to[out=0,in=0] ++(0,-2.5) -- ++(0,0.5);
\fill[looseness=1.75, gray!15] (-5,0) to[out=180,in=180] ++(0,1.5) -- ++(5,0) -- ++ (0,0.5) -- ++(-5,0) to[out=180,in=180] ++(0,-2.5) -- ++(0,0.5);
\begin{scope}
\clip(-5,-0.6) rectangle (5,0.1);
\shade [left color=gray!15, right color=gray!0] (-5.01,0) rectangle (-4.75,-0.5);
\shade [left color=gray!0, right color=gray!15] (5.01,0) rectangle (4.75,-0.5);
\foreach \p in { (-5.002,0), (-5.002,-0.5), (4.75,0), (4.75,-0.5) } { (\draw[line width=1.2pt] \p -- ++(0.252,0); }
\end{scope}
\draw[line width=1.2pt,looseness=1.75] (5,0) to[out=0,in=0] ++(0,1.5) -- ++(-10,0) to[out=180,in=180] ++(0,-1.5);
\draw[line width=1.2pt,looseness=1.75] (5,-0.5) to[out=0,in=0] ++(0,2.5) -- ++(-10,0) to[out=180,in=180] ++(0,-2.5);
\draw[very thick, Green, looseness=1.75] (-5,-0.25) to[out=180,in=180] ++(0,2) -- ++(10,0) to[out=0,in=0] ++(0,-2);
\draw[very thick, Green] (-5,-0.25) -- ++(1,0) ++(0.25,0) -- ++(0.25,0) ++(0.75,0) \foreach \i in {1,...,4} { -- ++(1,0) ++(0.25,0) } -- ++(0.25,0) ++(0.75,0) -- ++(1,0) ++(0.25,0) -- ++(0.5,0);

\foreach \x/\dtsign in {-4.75/1,-2.5/1,-1.25/-1,0/-1,1.25/1,3.5/1} {
  \ifthenelse{\dtsign=1}{\def\dtcolor{blue}}{\def\dtcolor{Green}}
  \draw[\dtcolor,looseness=1.75] (\x,0) ++(0.125,0) -- ++(0,0.75) to[out=90,in=90] ++(0.75,0) -- ++(0,-1.5) to[out=270,in=270] ++(-0.75,0) -- ++(0,1);
  \draw[line width=1.2pt] (\x,0) ++(-0.05,0) -- ++(0.05,0) -- ++(0,0.75) coordinate (topbend) ++(0.25,0) -- ++(0,-0.75) -- ++(0.5,0) ++(0.25,0) -- ++(0.25,0) coordinate (rightbend);
  \draw[link] (topbend) ++(0.25,0) to[out=90,in=90] ++(0.75,0) -- ++(0,-0.5) -- ++(0,-1) to[out=270,in=270] ++(-0.75,0) -- ++(0,0.25) -- ++(0.25,0) coordinate (se-bend);
  \draw[link] (topbend) to[out=90,in=90] ++(0.75,0) -- ++(0,-0.5) -- ++(0,-1) to[out=270,in=270] ++(-0.75,0) coordinate (sw-bend);
  \draw[line width=1.2pt] (sw-bend) -- ++(0,0.25) -- ++(-0.05,0);
  \draw[line width=1.2pt] (se-bend) -- ++(0.25,0) ++(0.25,0) -- ++(0.25,0);
  \draw[link] (rightbend) ++ (-0.25,0.25) -- ++(0,-1) ++(-0.25,0) -- ++(0,1);
}
\node at (-3.125,-0.25) {$\cdots$};
\node at (2.875,-0.25) {$\cdots$};
\draw[line width=1.2pt] (-2.5,0) -- ++(-0.25,0) (-2.5,-0.5) -- ++(-0.25,0);
\draw[line width=1.2pt] (3.25,-0) -- ++(0.25,0) (3.25,-0.5) -- ++(0.25,0);
\draw[line width=0.8pt, black, -Stealth] (2.05,1.5) -- ++(0.1,0);
\draw[line width=0.8pt, black, -Stealth] (2,2) -- ++(-0.1,0);

\draw[link,double=red] (-1.25,0) ++ (0.25,0.75) -- ++(0,-0.75) -- ++(0.5,0) ++(0.25,0) -- ++(0.25,0) -- ++(0,0.75) to[out=90,in=90] ++(0.75,0) -- ++(0,-0.5) -- ++(0,-1) to[out=270,in=270] ++(-0.75,0) -- ++(0,0.25) -- ++(-0.25,0) ++(-0.25,0) -- ++(-0.5,0) -- ++(0,-0.25) to[out=270,in=270] ++(0.75,0) coordinate (part2);
\draw[link,double=red] (part2) -- ++(0,1) -- ++(0,0.5) to[out=90,in=90] ++(-0.75,0) -- ++(0,-0.1);
\node[red] at (-0.125,-1.5) {$B_0 = -L_0$};
\draw[link] (-1.25,0) ++(0,-0.75) to[out=270,in=270] ++(0.75,0) -- ++(0,1.5) to[out=90,in=90] ++(-0.75,0);
\end{scope}

\begin{scope}[scale=1.2,yshift=-6.5cm]
\path (1.5,0) arc (0:180:1.5 and 0.25) -- ++(0,0.25) coordinate (b1) coordinate (h1) to[out=90,in=0,looseness=1.5] ++(-0.5,0.25) arc (270:90:0.05 and 0.1) to[out=0,in=270,looseness=1.5] ++(0.5,0.25) coordinate (b2) coordinate (i1) coordinate (h2) to[out=90,in=-20,looseness=1.1] ++(-0.5,0.25) to[out=160,in=160,looseness=1.25] ++(70:0.2) to[out=-20,in=-110,looseness=1.75] ++(0.6,0.3) coordinate (i2) arc (160:150:1.5) coordinate (b3) arc (150:140:1.5) coordinate (h3) to[out=50,in=310,looseness=1.25] ++(-0.3,0.5) to[out=130,in=130,looseness=1.5] ++(0.15,0.15) to[out=310,in=230,looseness=1.25] ++(0.5,-0.3) coordinate (i3) to[out=50,in=190,looseness=0.75] coordinate[midway] (b4) (-0.35,2.75) coordinate (h4) to[out=10,in=270,looseness=1.5] ++(0.25,0.6) arc (180:0:0.1 and 0.075) to[out=270,in=170,looseness=1.5] ++(0.25,-0.6) coordinate (i4) to[out=-10,in=130,looseness=0.75] coordinate[midway] (b5) ++(0.5,-0.3) to[out=310,in=40] ++(0.3,0.3);
\path let \p1 = (i3) in coordinate (r3) at (-\x1,\y1);

\draw[thick,fill=gray!10] (1.5,0) arc (0:180:1.5 and 0.2) -- ++(0,0.25) coordinate (b1) to[out=90,in=0,looseness=1.5] ++(-0.5,0.25) arc (270:90:0.05 and 0.1) coordinate[midway] (bl1) to[out=0,in=270,looseness=1.5] ++(0.5,0.25) coordinate (b2) to[out=90,in=-20,looseness=1.1] ++(-0.5,0.25) to[out=160,in=160,looseness=1.25] coordinate[midway] (bl2) ++(70:0.2) to[out=-20,in=-110,looseness=1.75] ++(0.6,0.3) arc (160:150:1.5) coordinate (b3) arc (150:140:1.5) to[out=50,in=310,looseness=1.25] ++(-0.3,0.5) to[out=130,in=130,looseness=1.5] ++(0.15,0.15) to[out=310,in=230,looseness=1.25] ++(0.5,-0.3) to[out=50,in=190,looseness=0.75] coordinate[midway] (b4) (-0.35,2.75) to[out=10,in=270,looseness=1.5] ++(0.25,0.6) arc (180:0:0.1 and 0.075) coordinate (blhalf) to[out=270,in=170,looseness=1.5] ++(0.25,-0.6) to[out=-10,in=130,looseness=0.75] coordinate[midway] (b5) (r3) to[out=-40,in=230] ++(0.5,0.3) to[out=50,in=50,looseness=1.5] ++(0.15,-0.15) to[out=230,in=130,looseness=1.25] ++(-0.3,-0.5) arc (40:30:1.5) coordinate (b6) arc (30:20:1.5) to[out=-70,in=200,looseness=1.75] ++(0.6,-0.3) to[out=20,in=20,looseness=1.5] coordinate[midway] (bl-2) ++(-70:0.2) to[out=200,in=90,looseness=1.1] ++(-0.5,-0.25) coordinate (b7) to[out=270,in=180,looseness=1.5] ++(0.5,-0.25) arc (90:-90:0.05 and 0.1) coordinate[midway] (bl-1) to[out=180,in=90,looseness=1.5] ++(-0.5,-0.25) coordinate (b8) -- ++(0,-0.25);
\draw[thick,red] (1.5,0) arc (0:180:1.5 and 0.2);

\coordinate (hole) at (0,1.25);
\draw[Green,very thick,densely dashed] (hole) ++(210:0.25) to[bend right=15] (b1);
\draw[blue,very thick,densely dashed] (hole) ++(180:0.25) to[bend right=15] (b2);
\draw[blue,very thick,densely dashed] (hole) ++(150:0.25) to[bend right=15] (b3);
\draw[blue,very thick,densely dashed] (hole) ++(105:0.25) to[bend right=15] (b4);
\draw[blue,very thick,densely dashed] (hole) ++(75:0.25) to[bend left=15] (b5);
\draw[blue,very thick,densely dashed] (hole) ++(30:0.25) to[bend left=15] (b6);
\draw[blue,very thick,densely dashed] (hole) ++(0:0.25) to[bend left=15] (b7);
\draw[Green,very thick,densely dashed] (hole) ++(-30:0.25) to[bend left=15] (b8);
\draw[fill=white] (hole) circle (0.25);

\draw[thin,-Stealth] (0,-0.75) -- (0,4.25);
\draw[thin,-latex] (0,3.85) node[left, outer sep=5pt] {$\tau$} ++(105:0.25 and 0.1) arc(105:440:0.25 and 0.1);

\node[left] at (bl1) {$B_1$};
\node[left] at (bl2) {$B_2$};
\node[left] at (b3) {\rotatebox{60}{$\cdots$}};
\node[right] at (blhalf) {$B_{n/2}$};
\node[right] at (b6) {\rotatebox{-60}{$\cdots$}};
\node[right] at (bl-2) {$B_{n-2}$};
\node[right] at (bl-1) {$B_{n-1}$};
\node[below,red] at (300:1.5 and 0.25) {$B_0$};

\draw[thick,fill=gray!15, fill opacity=0.75] (1.5,0) arc (360:180:1.5 and 0.25) -- ++(0,0.25) coordinate (b1) to[out=90,in=0,looseness=1.5] ++(-0.5,0.25) arc (-90:90:0.05 and 0.1) to[out=0,in=270,looseness=1.5] ++(0.5,0.25) coordinate (b2) to[out=90,in=-20,looseness=1.1] ++(-0.5,0.25) to[out=340,in=340,looseness=1] ++(70:0.2) to[out=-20,in=-110,looseness=1.75] ++(0.6,0.3) arc (160:150:1.5) coordinate (b3) arc (150:140:1.5) to[out=50,in=310,looseness=1.25] ++(-0.3,0.5) to[out=310,in=310,looseness=1] ++(0.15,0.15) to[out=310,in=230,looseness=1.25] ++(0.5,-0.3) to[out=50,in=190,looseness=0.75] coordinate[midway] (b4) (-0.35,2.75) to[out=10,in=270,looseness=1.5] ++(0.25,0.6) arc (180:360:0.1 and 0.075) to[out=270,in=170,looseness=1.5] ++(0.25,-0.6) to[out=-10,in=130,looseness=0.75] coordinate[midway] (b5) (r3) to[out=-40,in=230] ++(0.5,0.3) to[out=230,in=230,looseness=1] ++(0.15,-0.15) to[out=230,in=130,looseness=1.25] ++(-0.3,-0.5) arc (40:30:1.5) coordinate (b6) arc (30:20:1.5) to[out=-70,in=200,looseness=1.75] ++(0.6,-0.3) to[out=200,in=200,looseness=1] ++(-70:0.2) to[out=200,in=90,looseness=1.1] ++(-0.5,-0.25) coordinate (b7) to[out=270,in=180,looseness=1.5] ++(0.5,-0.25) arc (90:270:0.05 and 0.1) to[out=180,in=90,looseness=1.5] ++(-0.5,-0.25) coordinate (b8) -- ++(0,-0.25);

\draw[fill=white] (0,1.25) coordinate (hole) circle (0.25);
\draw[thin] (hole) ++(0,-0.25) -- ++(0,0.5); 
\draw[Green, very thick] (hole) circle (0.5) ++(270:0.5) node[below, black, inner sep=2pt] {\footnotesize$a$};
\draw[Green,very thick] (hole) ++(210:0.25) to[bend left=15] coordinate[midway] (b1-label) (b1);
\path (b1-label) ++(0.15,-0.15) node {\footnotesize$b_0$};
\draw[blue,very thick] (hole) ++(180:0.25) to[bend left=15] coordinate[pos=0.85] (b2-label) (b2);
\path (b2-label) ++(0,-0.2) node {\footnotesize$b_1$};
\draw[blue,very thick] (hole) ++(150:0.25) to[bend left=15] (b3);
\draw[blue,very thick] (hole) ++(105:0.25) to[bend left=15] (b4);
\draw[blue,very thick] (hole) ++(75:0.25) to[bend right=15] (b5);
\draw[blue,very thick] (hole) ++(30:0.25) to[bend right=15] (b6);
\draw[blue,very thick] (hole) ++(0:0.25) to[bend right=15] (b7);
\draw[Green,very thick] (hole) ++(-30:0.25) to[bend right=15] coordinate[midway] (b-1-label) (b8);
\path (b-1-label) ++(-0.15,-0.15) node {\footnotesize$b_{n-1}$};

\draw[very thick,red] (1.5,0) arc (360:180:1.5 and 0.25);
\end{scope}
\end{tikzpicture}
\caption{The fiber surface $F$, viewed abstractly, with monodromy a composition of Dehn twists about the indicated curves.  We perform negative Dehn twists about green curves, and positive Dehn twists about blue curves.}
\label{fig:abstract}
\end{figure}
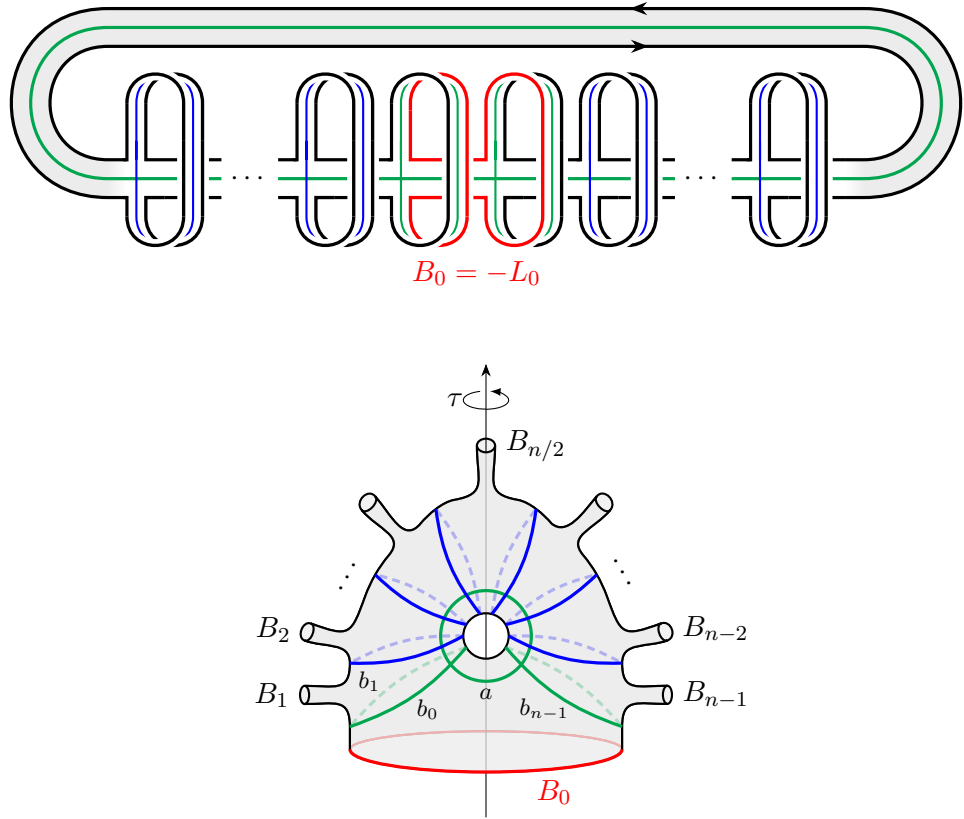

Let $\nu(L_0),\dots,\nu(L_{n-1})$ be disjoint closed tubular neighborhoods of the link components. There are two natural oriented meridian-longitude pairs $(\mu_i,\lambda_i)$ and $(\mu_i',\lambda_i')$ on the boundary torus $-\partial \nu(L_i) = \partial (S^3\setminus \nu(L_i))$ for each $i$. Namely, $\lambda_i$ is the longitude determined by the disk bounded by the oriented component $L_i$, while $\lambda_i'$ is the longitude determined by the corresponding oriented boundary component $B_i$ of the fiber surface $F$, as described in \S\ref{ssec:flowsdehn}. The meridians $\mu_i$ and $\mu_i'$ are then oriented according to
\[ \lambda_i\cdot\mu_i = \lambda_i'\cdot \mu_i' = 1 \]
on $\partial (S^3\setminus \nu(L_i))$. Since the orientation of $B_0$ is opposite that of $L_0$, we have
\begin{equation}\label{eqn:mu}\mu_i' = \begin{cases}
 -\mu_0,&i=0\\
 \mu_i,&i\neq 0.\\
\end{cases}\end{equation}
Moreover, it is clear from Figure \ref{fig:F} that
\begin{equation}\label{eqn:lambda}\lambda_i' = \begin{cases}
 -\lambda_0+2\mu_0,&i=0\\
 \lambda_i,&i=1 \textrm{ or }n-1\\
  \lambda_i+2\mu_i,&\textrm{otherwise.}\\
\end{cases}\end{equation}

Since $\mathbb{L}_n$ is hyperbolic, the monodromy $h$ is freely isotopic to a pseudo-Anosov homeomorphism. We would like to understand the degeneracy slopes of the corresponding suspension pseudo-Anosov flow on the link complement. As discussed in \S\ref{ssec:flowsdehn}, it suffices to determine the fractional Dehn twist coefficient \[c_{B_i}(h)\in\Q\] of $h$ at the boundary component $B_i$, for each $i = 0,\dots, n-1$.

\begin{proposition}
\label{prop:fdtc} We have $c_{B_0}(h) = -1/4$ and $c_{B_i}(h) = 0$ for all $i\neq 0$.
\end{proposition}

\begin{proof}
Suppose first that $i\neq 0$. It is easy to find properly embedded arcs $\alpha_i$ and $\beta_i$ on $F$ with at least one endpoint on $B_i$ such that $h(\alpha_i)$ is to the right of $\alpha_i$ and $h(\beta_i)$ is to the left of $\beta_i$ at $B_i$. For example, we can take $\alpha_i$ to be an arc that intersects one of $b_1,\dots,b_{n-2}$ in a single point and avoids the other curves, and $\beta_i$ to be an arc that intersects $a$ in a single point and avoids the other curves. In other words, $h$ is neither right-veering nor left-veering at $B_i$. It follows that $c_{B_i}(h) = 0$. 

Let us therefore focus on the case $i=0$. It is helpful to note that $h$ commutes with the involution $\tau$ shown in Figure \ref{fig:abstract}.
Let \[\bar F = F/\tau \,\,\, \textrm{ and }\,\, \,\bar h = h/\tau\,\,\, \textrm{ and }\,\,\, \bar B_i = B_i/\tau = B_{n-i}/\tau.\] Then $\bar F$ is a planar surface with boundary \[\partial \bar F = \bar B_0 \cup \dots \cup \bar B_{k=n/2}.\] Note that $F$ is the double cover of $\bar F$ branched along the points $p_1$ and $p_2$ shown in Figure \ref{fig:tildeF}.
\begin{figure}
\begin{tikzpicture}[scale=1.25]
\draw[red,very thick, fill=gray!10, even odd rule] (0,0) ellipse (3.5 and 1.5);
\node[red,above] at (30:3.5 and 1.5) {$\bar{B}_0$};
\draw[very thick, fill=white] \foreach \x in {-2.5,-1.5,0.5} { (\x,0) circle (0.15) };
\node at (-0.75,0) {$\cdots$};
\draw[Green,very thick] (1.5,0) -- (2.5,0);
\draw[very thick, fill=black] (1.5,0) coordinate (p1) circle (0.05) node[left, inner sep=1pt] {\footnotesize$p_1$};
\draw[very thick, fill=black] (2.5,0) coordinate (p2) circle (0.05) node[right, inner sep=2pt] {\footnotesize$p_2$};
\draw[thin,-latex] (-2.25,0.65) node[above,inner sep=1pt] {$\bar{B}_1$} -- ++(-0.2,-0.48);
\draw[thin,-latex] (-1.25,0.65) node[above,inner sep=1pt] {$\bar{B}_2$} -- ++(-0.2,-0.48);
\node[above] at (-0.175,0.65) {\footnotesize$\cdots$};
\draw[thin,-latex] (0.9,0.65) node[above,inner sep=1pt] {$\bar{B}_k$} -- ++(0,-0.65) -- ++(-0.2,0);
\draw[thin,-latex] (2.75,-0.5) node[right, inner sep=1pt] {\footnotesize$x$} to[bend left=30] (2.25,-0.05);
\draw[blue,very thick] (1.5,0.25) arc (90:-90:0.25) -- ++(-1,0) arc (270:90:0.25) node[left,midway,black,inner sep=0.5pt] {\footnotesize$y_{k-1}$} -- ++(1,0);
\draw[blue,very thick] (1.5,0.4) arc (90:-90:0.4) -- ++(-2.85,0) arc (270:90:0.4) node[left,midway,black,inner sep=0pt] {\footnotesize$y_1$} -- ++(2.85,0);
\draw[Green,very thick] (1.5,0.55) arc (90:-90:0.55) -- ++(-4,0) arc (270:90:0.55) node[left,midway,black,inner sep=0pt] {\footnotesize$y_0$} -- ++(4,0);
\draw[orange,very thick] (3.5,0) to[out=180,in=60,looseness=1.5] node[above,pos=0.6,inner sep=2pt] {$\alpha$} (1.5,0);
\draw[orange,very thick] (3.5,0) to[out=180,in=10,looseness=1.1] (1.25,-1) to[out=190,in=315,looseness=1.25] node[above,pos=0.4,inner sep=2pt] {$\beta$} (-2.5,0);
\draw[very thick, fill=white] (-2.5,0) circle (0.15);
\draw[very thick, fill=black] (p1) circle (0.05);
\end{tikzpicture}
\caption{The quotient $\bar{F}$ of $F$ by the involution $\tau$.}
\label{fig:tildeF}
\end{figure}
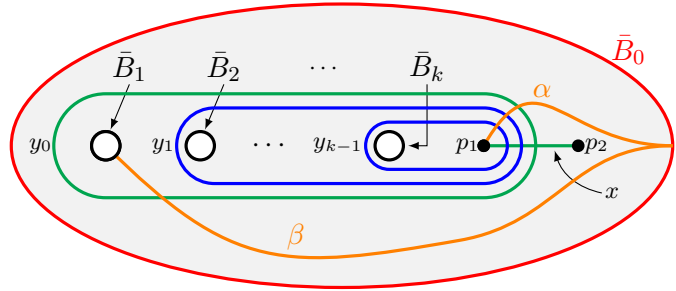
Moreover, $\bar h$ is given by the composition
\[\bar h = D_{y_0}^{-1}\circ D_{y_1}\circ\dots \circ D_{y_{k-1}}\circ \sigma_x^{-1}\]
of Dehn twists with a negative half-twist $\sigma_x^{-1}$ along the arc $x$ from $p_1$ to $p_2$ shown in the figure. We consider this quotient mostly because it is easier to visualize the dynamics of $\bar h$ on this planar surface. Observe that
\[c_{B_0}(h) = -1/4 \iff c_{\bar B_0}(\bar h) = -1/2 \iff c_{\bar B_0}(\bar h^2) = -1 \iff c_{\bar B_0}(\bar h^2\circ D_\delta) = 0, \]
where $\delta$ is a curve in $\bar F$ parallel to $\bar B_0$. For the latter, it suffices to show that $\bar h^2\circ D_\delta(\alpha)$ sends some arc $\alpha$ to the right at $\bar B_0$ and another arc $\beta$ to the left at $\bar B_0$. This is true of the arcs $\alpha$ and $\beta$ in Figure \ref{fig:tildeF}, as indicated in Figure \ref{fig:leftright}.
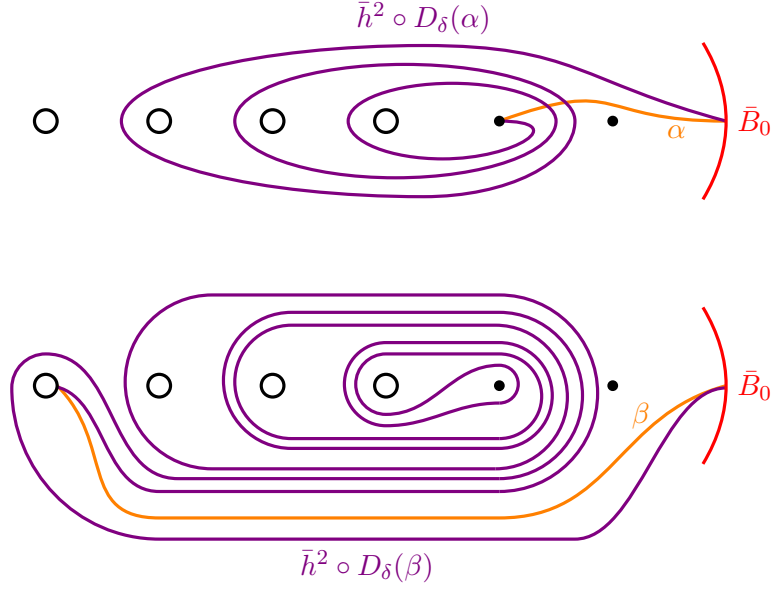
\begin{figure}
\begin{tikzpicture}
\begin{scope}
\draw[orange,very thick] (9,0) to[out=180,in=20,looseness=1.5] node[below,pos=0.15,inner sep=2pt] {$\alpha$} (6,0);
\draw[violet,very thick] (9,0) to[out=165,in=0,looseness=1.25] (5,1) node[above] {$\bar{h}^2 \circ D_\delta(\alpha)$} arc (90:270:4 and 1) arc (270:360:2 and 1) arc (0:180:2.25 and 0.75) arc (180:360:2.125 and 0.75) arc (0:180:1.375 and 0.5) arc (180:270:1.25 and 0.5) to[out=0,in=0,looseness=2.75] (6,0);
\draw[very thick, fill=white] \foreach \x in {0,...,3} { (1.5*\x,0) circle (0.15) };
\draw[very thick, fill=black] (4.5,0) \foreach \x in {1,2} { ++(1.5,0) coordinate (p\x) circle (0.05) };
\draw[very thick, red] ++(-15:9 and 4) arc (-15:15:9 and 4) node[right,midway] {$\bar{B}_0$};
\end{scope}
\begin{scope}[yshift=-3.5cm]
\draw[orange,very thick] (9,0) to[out=195,in=0,looseness=1.1] node[above,pos=0.35,inner sep=2pt] {$\beta$} (6,-1.75) -- ++(-4.5,0) to[out=180,in=315,looseness=1.25] (0.15,0);
\begin{scope}[every path/.append style={violet, very thick}]
\draw (6,-0.7) coordinate (h1b) arc (-90:90:0.56) to[out=180,in=0] ++(-1.5,0) arc (90:270:0.4) to[out=0,in=180] ++(1.5,0.65) arc (90:-90:0.25) coordinate (h1t);
\draw (h1t) to[out=180,in=0] ++(-1.5,-0.3) arc (270:90:0.55) -- ++(1.5,0) arc (90:-90:0.7) -- ++(-2.75,0) arc (270:90:0.9) -- ++(2.75,0) arc (90:-90:1.1) coordinate (h2b);
\draw (h1b) -- ++(-2.75,0) arc (270:90:0.75) -- ++(2.7,0) arc (90:-90:0.95) coordinate (h2t);
\draw (h2t) -- ++(-3.75,0) arc (270:90:1.15) -- ++(3.8,0) arc (90:-90:1.3) coordinate (end-beta);
\draw (end-beta) -- ++(-4.5,0) to[out=180,in=0] (0,0);
\draw (h2b) -- ++(-4.25,0) to[out=180,in=0,looseness=1] ++(-1.75,1.65) arc (90:180:0.45) to[out=270,in=180] ++(1.95,-2) -- node[below] {$\bar{h}^2 \circ D_\delta(\beta)$} ++(5.5,0) to[out=0,in=180,looseness=0.75] ++(2,2);
\end{scope}
\draw[very thick, fill=white] \foreach \x in {0,...,3} { (1.5*\x,0) circle (0.15) };
\draw[very thick, fill=black] (4.5,0) \foreach \x in {1,2} { ++(1.5,0) coordinate (p\x) circle (0.05) };
\draw[very thick, red] ++(-15:9 and 4) arc (-15:15:9 and 4) node[right,midway] {$\bar{B}_0$};
\end{scope}
\end{tikzpicture}
\caption{The images of $\alpha$ and $\beta$ under $\bar{h}^2 \circ D_\delta$.}
\label{fig:leftright}
\end{figure}
\end{proof}

Let $d_i$ be the degeneracy slope of the  pseudo-Anosov flow at the boundary component $L_i$, viewed as an integer multiple of an oriented curve on $\partial (S^3\setminus \nu(L_i))$. Each $d_i$ is determined by $c_{B_i}(h)$ together with the number  of prongs at $B_i$ of the stable foliation of the pseudo-Anosov representative of $h$, as in \S\ref{ssec:flowsdehn}.
If the stable foliation has interior singularities, call them $x_1,\dots, x_m$. Let $q_i$ denote the number of prongs at $B_i$, and let $s_i$ denote the number of prongs at $x_i$. 
\begin{lemma}
\label{lem:prongs}
We have $1\leq q_i \leq n+1$ for each $i$, and $3 \leq s_j \leq n+2$ for each $j$.
\end{lemma}

\begin{proof}
Standard Euler characteristic arguments tell us that
\[\sum_{i=0}^{n-1}(2-q_i) + \sum_{j=1}^m (2-s_j) = \chi(\hat F) = 0,\]
where $\hat F$ is the closed torus obtained from capping off the boundary components. Since each $s_i \geq 3$, we have
\[ \sum_{i=0}^{n-1}(2-q_i) = 2n-\sum_{i=0}^{n-1} q_i \geq 0. \]
Since each $q_i \geq 1$, we have
\[ 2n \geq \sum_{i=0}^{n-1} q_i \geq n-1 + q_i \]
for each $i$. Therefore, each $q_i\leq n+1$ as desired. Similarly, since each $q_i\geq 1$, we have
\[ \sum_{j=1}^m (s_j -2)=\sum_{i=0}^{n-1}(2-q_i) \leq n, \]
which implies for each $j$ that $s_j\leq n+2$ as desired.
\end{proof}

We can now say the following about degeneracy slopes with respect to the $(\mu_i,\lambda_i)$ coordinates.

\begin{lemma}
\label{lem:degeneracy} We have \[d_i = 
\begin{cases}
\ell_i\mu_i,& i\neq 0\\
\ell_0(-6\mu_0 + \lambda_0),& i = 0,
\end{cases}\]
where the $\ell_i$ are integers satisfying $1\leq \ell_i \leq n+1$, for each $i = 0,\dots, n-1$.
\end{lemma}

\begin{proof}
For $i\neq 0$, the fact that $c_{B_i}(h) = 0$ from Proposition \ref{prop:fdtc} implies, via the formula \eqref{eqn:degfdtc}, that \[d_i = q_i\mu_i' = q_i\mu_i.\] For the proposition, we just let $\ell_i = q_i$ in this case. 
Now suppose $i = 0$. The fact that $c_{B_0}(h) = -1/4$ from Proposition~\ref{prop:fdtc} implies that $q_0$ is a positive multiple of 4, and that
\[ d_0 = q_0(4\mu_0'-\lambda_0')/4 = q_0(-6\mu_0+\lambda_0)/4, \]
again via \eqref{eqn:degfdtc}, where the second equality follows from \eqref{eqn:mu} and \eqref{eqn:lambda}. So we just let $\ell_0 = q_0/4$.
\end{proof}

\begin{proposition}
\label{prop:distinctflows}
Let $M>8n\geq32$ be an odd integer, and let $r_i= M^{i+1}/4$. Then \[\mathbb{L}_n(r_0,\dots,r_{n-1})\] has $n$ distinct pseudo-Anosov flows, no two of which are orbit equivalent.
\end{proposition} 

\begin{proof}
Let $(a_0,\dots,a_{n-1})$ be an $n$-tuple of rational numbers. If the distance $\Delta(a_i,d_i)$ between the surgery slope $a_i$ and the degeneracy slope $d_i$ is at least $2$ for each $i=0,\dots,n-1$, then the pseudo-Anosov flow on the link complement extends to a pseudo-Anosov flow on
\[ \mathbb{L}_n(a_0,\dots,a_{n-1}) \]
via Fried surgery (recall per Remark~\ref{rmk:smoothtopological} that we implicitly have in mind the genuine pseudo-Anosov flow orbit equivalent to the topological flow obtained via Fried surgery).  Here the cores $\gamma_0,\dots,\gamma_{n-1}$ of the surgery solid tori are closed orbits, where the number of prongs at $\gamma_i$ is given by $\Delta(a_i,d_i)$, as described in \S\ref{ssec:flowsdehn}. All other singular orbits come from suspending the interior singularities $x_1,\dots,x_m$ of the stable foliation of the pseudo-Anosov representative of $h$ (if there are any), and the number of prongs at the suspension of $x_i$ is given by \[s_i\leq n+2 \leq M,\] where the first inequality comes from Lemma \ref{lem:prongs}.

Now fix any integer $0\leq k\leq n-1$ and let $a_i = r_{i+k}$, with subscripts taken mod ${n}$. Note that
\[\mathbb{L}_n(a_0,\dots,a_{n-1}) \cong \mathbb{L}_n(r_0,\dots,r_{n-1}),\]
by the symmetry of $\mathbb{L}_n$ under rotation of its components. We have
\[ \Delta(a_i, d_i) = \Delta(r_{i+k}, d_i) = |(M^{[i+k]+1}\mu_i +4\lambda_i) \cdot d_i|, \]
where $[i+k]$ refers to the unique representative in $\{0,\dots,n-1\}$ of the mod $n$ residue class of $i+k$. By Lemma \ref{lem:degeneracy}, this is equal to
\[ 4\ell_i\leq 4n+4<M \]
when $i\neq 0$, and to
\[ |(M^{k+1}\mu_0 +4\lambda_0) \cdot \ell_0(-6\mu_0+\lambda_0)| = \ell_0(M^{k+1} + 24) \geq M^{k+1} + 24>M \]
when $i=0$. Since these distances are all at least $4\geq 2$, the pseudo-Anosov flow on the link complement extends to a pseudo-Anosov flow on $\mathbb{L}_n(a_0,\dots,a_{n-1})$.  Moreover, $\gamma_0$ is the unique singular orbit of this flow with the most prongs, namely $\ell_0(M^{k+1}+24)$. The fact that the numbers $M^{k+1} + 24$ are distinct as $k$ ranges from $0$ to $n-1$ then shows that the induced flows on $\mathbb{L}_n(r_0,\dots,r_{n-1})$ are mutually orbit inequivalent as $k$ ranges from $0$ to $n-1$.
\end{proof}
 
\begin{remark}
\label{rmk:int}
We record here the observation in the above proof that $\Delta(r_i,d_i) \geq 4\geq 3$ for each $i$.
\end{remark}
 
\begin{remark}
Although we will not need it, one can show by computing an invariant train track for $h$ that the stable foliation of its pseudo-Anosov representative has no interior singularities, and that $s_0 = 4$, $s_1 = s_{n-1} = 1$, and $s_i = 2$ for all other $i$. Then $\ell_0 = 1$, $\ell_1 = \ell_{n-1} = 1$, and $\ell_i = 2$ for all other $i$, from which it follows that Proposition~\ref{prop:distinctflows} holds for any odd $M\geq 9$.
\end{remark}

\begin{lemma}
\label{lem:L-space} The manifolds in Proposition~\ref{prop:distinctflows} have first homology of odd order, which grows to infinity as $M$ does, and they are L-spaces when $M$ is sufficiently large.
\end{lemma}

\begin{proof}
The claim about $H_1(\mathbb{L}_n(r_0,\dots,r_{n-1});\Z)$ is a special case of the following more general fact: if $\mathbb{L} = L_0 \cup \dots \cup L_{n-1}$ is an $n$-component link, and we write $a_i = p_i/q_i$ in lowest terms for $i=0,\dots,n-1$, then $H_1(\mathbb{L}(a_0,\dots,a_{n-1});\Z)$ has odd order whenever all of the $q_i$ are even.  Indeed, if we let $\ell_{ij} = \operatorname{lk}(L_i,L_j)$, then this homology is presented by the matrix
\[ A = \begin{pmatrix}
p_0 & q_0 \ell_{01} & q_0 \ell_{02} & \dots & q_0 \ell_{0,n-1} \\
q_1\ell_{10} & p_1 & q_1 \ell_{12} & \dots & q_1 \ell_{1,n-1} \\
q_2\ell_{20} & q_2 \ell_{21} & p_2 & \dots & q_2 \ell_{2,n-1} \\
\vdots & \vdots & \vdots & \ddots & \vdots \\
q_{n-1}\ell_{n-1,0} & q_{n-1} \ell_{n-1,1} & q_{n-1}\ell_{n-1,2} & \dots & p_{n-1}
\end{pmatrix}, \]
and each $p_i$ is odd since it is coprime to $q_i$, so we have $A \equiv I_n \pmod{2}$ and thus
\[ \left| H_1(\mathbb{L}(a_0,\dots,a_{n-1});\Z) \right| = |{\det}(A)| \equiv \det(I_n) = 1 \pmod{2}. \]
In the case of $\mathbb{L}_n(r_0,\dots,r_{n-1})$, we have $r_i = M^{i+1}/4$ in lowest terms since $M$ is odd, so $p_i = M^{i+1}$ and $q_i=4$ and the claim follows.

The order of this homology can also be bounded from below using Ostrowski's inequality \cite{ostrowski}, which says that if each difference
\[ h_i = |p_i| - \sum_{j\neq i} |q_i \ell_{ij}| \]
between the diagonal entry and the other entries in the $i$th row of $A$ is strictly positive, then $|{\det}(A)| \geq h_0h_1\dots h_{n-1}$.  The positivity condition is equivalent to $|a_i| > \sum_{j\neq i} |\ell_{ij}|$ for each $i$, and for $\mathbb{L}_n(r_0,\dots,r_{n-1})$ this is simply $|r_i| > 2$; since each $r_i = M^{i+1}/4$ is greater than $2$, we have
\[ |H_1(\mathbb{L}_n(r_0,\dots,r_{n-1});\Z)| \geq \prod_{i=0}^{n-1} (M^{i+1}-8), \]
which is an increasing, unbounded function of $M>8$.

For the L-space claim, Liu \cite[Example~3.15]{liu-lspace-links} showed that for each $n\geq 3$ the link $\mathbb{L}_n$ is an L-space link, meaning that there is some constant $C = C(\mathbb{L}_n)$ such that $\mathbb{L}_n(a_0,\dots,a_{n-1})$ is an L-space for any collection of integers $a_i \geq C$.  In Proposition~\ref{prop:lspace-rational} we show that this implies the existence of $C' = C'(\mathbb{L}_n)$ such that $\mathbb{L}_n(a_0,\dots,a_{n-1})$ is an L-space for any rational slopes $a_i \geq C'$, and then for $Y = \mathbb{L}_n(\frac{M}{4},\dots,\frac{M^n}{4})$ as above it suffices to take $M > 4C'$.
\end{proof}

\begin{lemma}
\label{lem:hyperbolic} The manifolds in Proposition~\ref{prop:distinctflows} are hyperbolic for $M$ sufficiently large.
\end{lemma}

\begin{proof}
The $n$-chain link $\mathbb{L}_n$ is hyperbolic for all $n\geq 3$ by \cite[Theorem~5.1(ii)]{neumann-reid}.  A theorem of Hodgson and Kerckhoff \cite[Theorem~5.12]{hodgson-kerckhoff} therefore says that each Dehn surgery on $\mathbb{L}_n$ is hyperbolic as long as the surgery slopes are sufficiently long.  Concretely, for each $i=0,\dots,n-1$ there is a set $S_i \subset \Q\cup\{\infty\}$ of size at most 114 such that $\mathbb{L}_n(a_0,\dots,a_{n-1})$ is hyperbolic as long as $a_i \not\in S_i$ for all $i$.  (In fact, since the symmetry group of $\mathbb{L}_n$ acts transitively on its components, we can take $S_0=\dots=S_{n-1}$.)  It then suffices to take $M > 4\max_i \left(\max S_i\right)$.
\end{proof}

\begin{lemma}
\label{lem:Lspacetaut} An L-space with odd-order first homology has no taut foliation.
\end{lemma}

\begin{proof}
Suppose $Y$ is an L-space with odd-order first homology. Then $Y$ cannot have a co-orientable taut foliation \cite{osz-genus,bowden,kazez-roberts}. Suppose that $Y$ has a taut foliation $\mathscr{F}$ that is not co-orientable. Fix a metric on $Y$ and consider the line bundle $\pi:L\to Y$ whose fiber at each point $p\in Y$ is the subspace of $T_pY$ normal to the leaf of $\mathscr{F}$ containing $p$. Let $Y'$ be the unit sphere bundle of $L$. Then $\pi: Y'\to Y$ is a double covering, and the lift $\pi^{-1}(\mathscr{F})$ to $Y'$ is a taut foliation with a natural co-orientation, given at the point $(p,v)\in Y'$ by $v$. Note that $Y'$ must be connected, since otherwise it would simply be two copies of $Y$, and a co-orientation on $\pi^{-1}(\mathscr{F})$ would give a co-orientation on $\mathscr{F}$. On the other hand, connected double covers of $Y$ correspond to index-2 subgroups of $\pi_1(Y),$ which do not exist when $H_1(Y;\Z)$ has odd order, a contradiction.
\end{proof}

\begin{proof}[Proof of Theorem \ref{thm:main}]
Proposition \ref{prop:distinctflows} provides quarter-integer fillings of $\mathbb{L}_n$ with $n$ orbit inequivalent pseudo-Anosov flows, for each even $n\geq 4$. Lemmas \ref{lem:L-space} and \ref{lem:hyperbolic} say that for $M$ sufficiently large, these fillings are hyperbolic L-spaces with odd-order first homology, and therefore do not have taut foliations by Lemma \ref{lem:Lspacetaut}. Finally, the fact in Lemma \ref{lem:L-space} that the order of first homology grows to infinity as $M$ does implies that these fillings produce infinitely many distinct 3-manifolds.
\end{proof}

\begin{proof}[Proof of Theorem \ref{thm:qg}]
This follows from Fenley's result in Theorem \ref{thm:fenley} that a pseudo-Anosov flow without perfect fits is quasigeodesic, combined with Theorem \ref{thm:main} and the fact that the pseudo-Anosov flows in Proposition \ref{prop:distinctflows} do not have perfect fits by Proposition \ref{prop:surgery-no-perfect-fits} and Remark \ref{rmk:int}.
\end{proof}

\begin{lemma}
\label{lem:birkhoff}
The slope $p\mu_i + q\lambda_i$ intersects the fiber and degeneracy slopes $\lambda_i'$ and $d_i$ with opposite signs, for any rational $p/q>2$.
\end{lemma}

\begin{proof}
This follows immediately from the calculations
\[(p\mu_i + q\lambda_i)\cdot \lambda_i' = 
\begin{cases}
p+2q,& i = 0\\
-p,& i = 1\textrm{ or } n-1\\
-p+2q,& \textrm{otherwise}
\end{cases}\] and 
\[(p\mu_i + q\lambda_i)\cdot d_i = 
\begin{cases}
\ell_0(-p-6q),& i = 0\\
\ell_iq,& \textrm{otherwise},
\end{cases}\]
which follow directly from Lemma \ref{lem:degeneracy} together with \eqref{eqn:mu} and \eqref{eqn:lambda}.
\end{proof}

\begin{proof}[Proof of Theorem \ref{thm:UT}]
For each even $n\geq 4$, Theorem~\ref{thm:main} furnishes via Proposition~\ref{prop:distinctflows} infinitely many hyperbolic L-spaces with $n$ orbit inequivalent pseudo-Anosov flows. These pseudo-Anosov flows  have negative Birkhoff sections per the discussion in \S\ref{ssec:flowsdehn}, since the slope \[r_i = M^{i+1}/4 > 2\] intersects $\lambda'_i$ and $d_i$ in opposite signs for each $i$, according to Lemma~\ref{lem:birkhoff}. The corresponding pseudo-Anosov flows on \[-\mathbb{L}_n(r_0,\dots,r_n)\] therefore have positive Birkhoff sections by Remark~\ref{rmk:posneg}. These  flows then give rise to $n$ universally tight positive contact structures $\xi_1,\dots,\xi_n$ on $-\mathbb{L}_n(r_0,\dots,r_n)$ per Corollary~\ref{cor:universallytight}. Moreover, the fact that these flows are orbit inequivalent implies via Proposition~\ref{prop:rsft} that these contact structures are mutually non-contactomorphic. Since the pseudo-Anosov flows are distinguished by the maximal number of prongs at a singular orbit, their lifts to any finite cover are also orbit inequivalent. This then implies that the lifts of $\xi_1,\dots,\xi_n$ are non-contactomorphic, again by Proposition~\ref{prop:rsft}.
\end{proof}

\begin{proof}[Proof of Theorem~\ref{thm:knot}] Suppose $K$ is a hyperbolic L-space knot of genus $g$, and let $r=p/q>4g$. Then $S^3_r(K)$ is hyperbolic by \cite{ni-exceptional}, and it is an L-space since $r\geq 2g-1$. Since $K$ is hyperbolic, it has pseudo-Anosov monodromy. Consider the associated suspension pseudo-Anosov flow on the complement of $K$. The degeneracy slope of this flow is of the form $a\mu + b\lambda$, where $a/b\leq 4g-2$. The distance between the degeneracy slope and $r$ is thus at least 3, so Proposition~\ref{prop:surgery-no-perfect-fits} says that Fried surgery produces a pseudo-Anosov flow on $S^3_r(K)$ with no perfect fits. By Corollary~\ref{cor:universallytight}, there is a corresponding universally tight contact structure.
\end{proof}

\appendix

\section{Large rational surgeries on L-space links} \label{sec:rational-surgeries}

A link $\mathbb{L} \subset S^3$ with $n$ components is called an \emph{L-space link} \cite{gorsky-nemethi} if there is some $C=C(\mathbb{L})$ such that for every choice of \emph{integers} $a_1,\dots,a_n \geq C$, the Dehn surgery \[ \mathbb{L}(a_1,\dots,a_n) \] is an L-space.  In fact, this is still true if we allow the $a_i$ to be rational.

\begin{proposition} \label{prop:lspace-rational}
Let $\mathbb{L}$ be an L-space link with $n$ components.  Then there is an integer $C'=C'(\mathbb{L})$ such that \[ \mathbb{L}(a_1,\dots,a_n) \] is an L-space for every choice of \emph{rational} slopes $a_1,\dots,a_n \geq C'$.
\end{proposition}

To set the stage, we let $\ell_{ij}$ be the linking number of the $i$th and $j$th components of $\mathbb{L}$.  We then define
\[ f(x_1,\dots,x_n) = \det\begin{pmatrix}
x_1 & \ell_{12} & \ell_{13} & \dots & \ell_{1n} \\
\ell_{21} & x_2 & \ell_{23} & \dots & \ell_{2n} \\
\ell_{31} & \ell_{32} & x_3 & \dots & \ell_{3n} \\
\vdots & \vdots & \vdots & \ddots & \vdots \\
\ell_{n1} & \ell_{n2} & \ell_{n3} & \dots & x_n
\end{pmatrix}, \]
which has the property that if we write $a_i = \frac{p_i}{q_i}$ in lowest terms, with $q_i > 0$, then
\begin{equation} \label{eq:homology-L-surgery}
\left| H_1(\mathbb{L}(a_1,\dots,a_n); \Z) \right| =  \left| q_1q_2\dots q_n \cdot f(a_1,\dots,a_n) \right|.
\end{equation}
Moreover, expanding by minors along the $i$th row shows that $f$ is an affine function in each of the $x_i$: we can write
\begin{equation} \label{eq:f-det-affine}
f(x_1,\dots,x_n) = g_i(x_1,\dots,\widehat{x_i},\dots,x_n) \cdot x_i + h_i(x_1,\dots,\widehat{x_i},\dots x_n),
\end{equation}
where the notation ``$\widehat{x_i}$'' means that we have omitted $x_i$ from the list of inputs.

\begin{lemma} \label{lem:determinant-positive}
There exists $M = M(\mathbb{L}) > 0$ such that if $x_i \geq M$ for all $i$, then $f(x_1,\dots,x_n) > 0$.
\end{lemma}

\begin{proof}
We let $k = \max_{i\neq j} |\ell_{ij}|+1$, and set $M = n! \cdot k$, which depends only on the linking matrix of $\mathbb{L}$.  Then we can write
\[ f(x_1,\dots,x_n) = x_1x_2\dots x_n + \sum_{\sigma \in \mathcal{S}_n\setminus\{1\}} (-1)^\sigma p_\sigma, \]
where $p_\sigma$ is the product over all $i$ of the $(i,\sigma(i))$-entries of the matrix whose determinant is $f$.  When $\sigma \neq 1$, at least one of these entries is not a diagonal entry, so if the non-diagonal entry is $\ell_{j,\sigma(j)}$ and all of the $x_i$ satisfy $x_i \geq M \geq k$, then we have
\[ \left| \frac{p_\sigma}{x_1x_2\dots x_n}\right| \leq \left| \frac{x_1\dots x_{j-1} \cdot \ell_{j,\sigma(j)} \cdot x_{j+1} \dots x_n}{x_1x_2\dots x_n} \right| = \frac{|\ell_{j,\sigma(j)}|}{x_j} < \frac{k}{M} = \frac{1}{n!}. \]
This works for all $n!-1$ of the summands where $\sigma \neq 1$, so we can now use the above expression to bound $f$ below (assuming again that $x_i \geq M$ for all $i$) by
\[ f(x_1,\dots,x_n) \geq x_1x_2\dots x_n \left(1 - (n!-1)\frac{1}{n!}\right) > 0. \qedhere \]
\end{proof}

\begin{remark}
We do not claim that $M=n!\cdot k$ is anywhere near optimal; it can certainly be improved by using Ostrowski's inequality, as in the proof of Lemma~\ref{lem:L-space}.
\end{remark}

\begin{proof}[Proof of Proposition~\ref{prop:lspace-rational}]
Since $\mathbb{L}$ is an L-space link, there is a constant $C$ such that $\mathbb{L}(a_1,\dots,a_n)$ is an L-space whenever the $a_i$ are integers greater than or equal to $C$.  We take $M$ as provided by Lemma~\ref{lem:determinant-positive}, and let $C' = \max(C,M)$.  Our goal is to show that $\mathbb{L}(a_1,\dots,a_n)$ is an L-space whenever all of the $a_i$ are \emph{rational} and satisfy $a_i \geq C'$.  Certainly this is true when all of the $a_i$ are integers, because $C' \geq C$.

Suppose that we have some rational slopes $a_1,\dots,a_n \geq C'$, written $a_i = \frac{p_i}{q_i}$ in lowest terms with $q_i > 0$, such that the manifold 
\[ Y = \mathbb{L}(a_1,\dots,a_n) \]
is not an L-space.  Then at least one of the $a_i$ must not be an integer; we let $d \in \{1,\dots,n\}$ be the maximal such index, so that $a_d$ is not an integer but $a_{d+1},\dots,a_n \in \Z$.  We suppose that among all such tuples $(a_1,\dots,a_n)$, we have chosen this one to first minimize the value of $d$, and then to minimize the denominator $q_d$ among all such tuples with this value of $d$.  Since $a_d \not\in \Z$ we have $q_d \geq 2$, and then by the proof of \cite[Proposition~4.3]{bs-concordance} (which works equally well in Heegaard Floer homology) there exists a pair of fractions
\begin{align*}
a'_d &= \frac{p'}{q'}, &
a''_d &= \frac{p''}{q''},
\end{align*}
with $q',q'' \geq 1$ and $\lfloor a_d \rfloor \leq a'_d,a''_d \leq \lceil a_d \rceil$, such that 
\begin{enumerate}
\item we have $(p_d,q_d) = (p'+p'', q'+q'')$, and \label{i:mediant}
\item if we let
\begin{align*}
Y' &= \mathbb{L}(a_1,\dots,a_{d-1}, a'_d, a_{d+1},\dots, a_n), \\
Y'' &= \mathbb{L}(a_1,\dots,a_{d-1}, a''_d, a_{d+1},\dots,a_n),
\end{align*}
then there is a surgery exact triangle of the form
\begin{equation} \label{eq:surgery-triad}
\cdots \to \hfhat(Y') \to \hfhat(Y) \to \hfhat(Y'') \to \cdots .
\end{equation}
\end{enumerate}
In particular, item \eqref{i:mediant} says that either $a'_d \in \Z$ or $2 \leq q' < q_d$, so by our assumption of minimality we see that $Y'$ must be an L-space, and likewise for $Y''$.  Therefore \eqref{eq:surgery-triad} tells us that
\begin{equation} \label{eq:hfhat-upper-bound}
\rank \hfhat(Y) \leq \rank \hfhat(Y') + \rank \hfhat(Y'') = |H_1(Y';\Z)| + |H_1(Y'';\Z)|.
\end{equation}

Now if we write $Q = q_1\dots q_n$, then \eqref{eq:homology-L-surgery} and Lemma~\ref{lem:determinant-positive} tell us that
\begin{align*}
|H_1(Y';\Z)| &= \tfrac{Q}{q_d} \cdot q' \cdot f(a_1,\dots,a_{d-1},a'_d,a_{d+1},\dots,a_n) \\
|H_1(Y'';\Z)| &= \tfrac{Q}{q_d} \cdot q'' \cdot f(a_1,\dots,a_{d-1},a''_d,a_{d+1},\dots,a_n),
\end{align*}
in which the right side of each equation is positive because $a'_d, a''_d \geq \lfloor a_d \rfloor \geq C' \geq M$.  Applying \eqref{eq:f-det-affine} to these equations, we have
\begin{align*}
|H_1(Y';\Z)| &= \tfrac{Q}{q_d} \cdot q' \cdot \left( g_d(a_1,\dots,\widehat{a_d},\dots,a_n) \cdot a'_d + h_d(a_1,\dots,\widehat{a_d},\dots a_n)\right) \\
&= \tfrac{Q}{q_d} \left( g_d(a_1,\dots,\widehat{a_d},\dots,a_n) \cdot p' + h_d(a_1,\dots,\widehat{a_d},\dots a_n) \cdot q' \right)
\end{align*}
and likewise
\[ |H_1(Y'';\Z)| = \tfrac{Q}{q_d} \left( g_d(a_1,\dots,\widehat{a_d},\dots,a_n) \cdot p'' + h_d(a_1,\dots,\widehat{a_d},\dots a_n) \cdot q'' \right), \]
and we add these last two equations together to get
\begin{align*}
|H_1(Y';\Z)|+|H_1(Y'';\Z)| &= \tfrac{Q}{q_d} \left( g_d(a_1,\dots,\widehat{a_d},\dots,a_n) (p'+p'') \right. \\
&\qquad\qquad\left. + h_d(a_1,\dots,\widehat{a_d},\dots a_n) (q'+q'') \right) \\
&= \tfrac{Q}{q_d} \left( g_d(a_1,\dots,\widehat{a_d},\dots,a_n) p_d + h_d(a_1,\dots,\widehat{a_d},\dots a_n) q_d \right) \\
&= Q \cdot \left( g_d(a_1,\dots,\widehat{a_d},\dots,a_n) a_d + h_d(a_1,\dots,\widehat{a_d},\dots a_n) \right) \\
&= q_1\dots q_n \cdot f(a_1,\dots,a_n) \\
&= |H_1(Y;\Z)|.
\end{align*}
Now combining this with the upper bound \eqref{eq:hfhat-upper-bound} gives
\[ \rank \hfhat(Y) \leq |H_1(Y;\Z)| = \chi\left( \hfhat(Y) \right), \]
so equality must hold throughout.  In particular this means that $\hfhat(Y)$ has rank $|H_1(Y;\Z)|$, so $Y$ is an L-space as well and we have a contradiction.
\end{proof}

\bibliographystyle{myalpha}
\bibliography{References}

\end{document}